\documentclass[reqno]{amsart}
\usepackage{amssymb,amsmath}
\usepackage{amsthm}
\usepackage{color,graphicx}
\usepackage{xcolor}
\usepackage{hyperref}
\usepackage{verbatim}

\usepackage{stmaryrd}

\usepackage{bbm}

\newcommand{\les}{\lesssim}

\newcommand{\lanx}{\langle x \rangle}

\newcommand{\R}{\mathbb R}

\newcommand{\p}{\partial}

\newcommand{\sgn}{\text{sgn}}

\numberwithin{equation}{section}

\newtheorem{theorem}{Theorem}[section]
\newtheorem{proposition}[theorem]{Proposition}
\newtheorem{remark}[theorem]{Remark}
\newtheorem{lemma}[theorem]{Lemma}
\newtheorem{corollary}[theorem]{Corollary}

\begin{document}

\title[The fKdV equation ]{On the decay of solutions for the negative fractional KdV equation }



\author[A. Cunha]{Alysson Cunha}
\address{Instituto de Matem\'atica e Estat\'istica(IME).
Universidade Federal de Goi\'as (UFG), Campus Samambaia, 131, 74001-970, Goi\^ania, Bra\-zil}
\email{alysson@ufg.br}

\author[O. Ria\~no]{Oscar Ria\~no}
\address{Departamento de Matem\'aticas, Universidad Nacional de Colombia, Ca\-rre\-ra 30 No. 45-03, 111321, Edificio Yu Takeuchi 404-209, Bogot\'a D.C., Colombia}
\email{ogrianoc@unal.edu.co}

\author[A. Pastor]{Ademir Pastor}
\address{Departamento  de  Matem\'atica,  Instituto  de  Matem\'atica,  Estat\'{\i}stica  e  Computa\c c\~ao  Cient\'{\i}fica(IMECC), Universidade Estadual de Campinas (UNICAMP), Rua S\'ergio Buarque de Holanda, 651,13083-859, Campinas, SP, Brazil}
\email{apastor@ime.unicamp.br}



\begin{abstract}
We explore the limits of fractional dispersive effects and their incidence in the propagation of polynomial weights. More precisely, we consider the fractional KdV equation when a differential operator of negative order determines the dispersion. We investigate what magnitude of weights and conditions on the initial data that allow solutions of the equation to persist in weighted spaces. As a consequence of our results, it follows that even in the presence of negative dispersion, it is still possible to propagate weights whose maximum magnitude is related to the dispersion of the equation. We also observe that our results in weighted spaces do not follow specific properties and limits that their counterparts with positive dispersion.
\end{abstract}
 
\maketitle

\section{Introduction}\label{introduction}

This paper is concerned with the initial value problem (IVP) associated with the negative dispersion fractional KdV equation (fKdV)
\begin{equation}\label{gbo}
\begin{cases}
\partial_t u-\partial_{x} D^{a+1} u+u\partial_xu=0, \;\;x,t\in\R, \quad a \in (-\frac52,-2), \\
u(x,0)=\phi(x),
\end{cases}
\end{equation}
where $D^{a+1}$ stands for the fractional derivative operator of order $a+1$ with respect to the variable $x$, which is defined via the Fourier transform as $D^{a+1} f=(|\xi|^{a+1}\widehat{f})^\vee$.

For different dispersions $a\in \mathbb{R}$, equation fKdV has been used as a test model to analyze the competition between dispersion and nonlinearity. When $a=1$, one obtains the widely studied Korteweg-de Vries equation, which appears in various contexts such as shallow-water waves with weakly non-linear restoring forces, see \cite{Miura1976}. In this work, we focus on instances where the operator $D^{a+1}$ is nonlocal, where we highlight the following cases:

$\bullet$ \emph{Dispersion $a=0$ in fKdV}. The resulting model is the Benjamin-Ono equation, which has been used to study waves in stratified deep water. Concerning well-posedness in Sobolev spaces $H^s(\mathbb{R})$, see \cite{AbBonaFellSaut1989, BurqPlanchon2008,IfrimTataru2019,IonescuKenig2007, Iorio1986, KillipLaurensVican2023,  KochTzvetkov2003, MolinetPilod2012,   Ponce1991, Tao2004}, and for persistence in weighted spaces, we refer the reader to \cite{Cunha2022,fp,Iorio1986}. See also \cite{dBO} regarding results on weighted spaces for a dissipative perturbation of \eqref{gbo}. 

$\bullet$ \emph{Dispersion $0<a<1$ in fKdV}. One obtains the dispersion generalized Benjamin-Ono equation, which has been used to model physical phenomena such as vorticity waves in the coastal zone \cite{NaumkinShishmarev1994, ShriraVoronovich1996}. See \cite{HerIonescuKenigKoch2010}, for well-posedness in $H^s(\mathbb{R})$  and \cite{FLP1} for well-posedness in weighted spaces. Furthermore, well-posedness in weighted spaces for a bidimensional version of \eqref{gbo} can be found in \cite{apl}. 

$\bullet$ \emph{Dispersion $-1<a<0$ in fKdV}. We obtain the low-order dispersive fractional KdV equation whose well-posedness in $H^s(\mathbb{R})$ was investigated in \cite{LinaresPilodSaut2014, MolinetPilodVento2018}, and in weighed spaces in \cite{Rianogbo}. 

$\bullet$ \emph{Dispersion $-2<a<-1$ in fKdV}. The case $a=-\frac{3}{2}$ has been proposed in the study of water waves in two dimensions with infinite depth, see \cite{Hur2012}. Concerning well-posedness in weighted spaces and further remarks on that equation, see \cite{KenigPilodPonceVega2020, Rianogbo}.  

$\bullet$ \emph{Dispersion $a=-2$ in fKdV}. In this case, we obtain the Burgers-Hilbert equation, introduced to study nonlinear constant-frequency waves \cite{BielloHunter2010, HunterIfrimTataruWong2015}. For well-posedness in weighted spaces, see \cite{Rianogbo}   

$\bullet$ The equation has also been studied in several variable settings. See \cite{HBO,HickmanLinaresRiano2019,Riano2020,RianoRoudenko2022,RianoRoudenkoYang2022,Schippa2020}, and references therein. 

We observe that dispersions $a\geq-2$ in \eqref{gbo} have been extensively studied in the literature. This paper seeks to push the dispersive effects to the limit and consider the extremes $a<-2$. Our goal is to characterize how the dispersion influences certain conditions on the initial data to establish for which values of $\theta>0$, the polynomial weight $|x|^{\theta}$ is propagated by solutions of \eqref{gbo}, i.e., we aim to determinate $\theta>0$ such that there exists a local solution  $u$ of \eqref{gbo} with $|x|^{\theta} u\in L^{\infty}([0,T];L^2(\mathbb{R}))$. This kind of question is well understood in the case $a>-2$, where it has been established, in conjunction with some unique continuation principles (UCP), that $\theta<\frac{7}{2}+a$ is the weight limit propagated by solutions of \eqref{gbo}, see \cite{FLP1,fp,Rianogbo} for details. However,  we will see that this limit on $\theta$ changes for negative dispersions $a<-2$, and extra conditions on the initial data are necessary to propagate weights.  Consequently, our results give a complete framework of the spatial decay of solutions of \eqref{gbo} in polynomial spaces when $a>-\frac{5}{2}$.

On the other hand, we recall that for the Cauchy problem \eqref{gbo} some local UCP and asymptotic at infinity UCP have been studied before. For more information on these questions, see \cite{LinaresPonce2022}. Actually, local UCP for various cases of dispersion $a\in \mathbb{R}$ in \eqref{gbo} have been obtained in \cite{KenigPilodPonceVega2020,KenigPonceVega2020}.  Concerning asymptotic at infinity UCP when $a\geq -2$, see references \cite{FLP1,FonsecaLinaresPonce2012,fp,Rianogbo}. However, when $a<-2$, although there exist results on local UCP, there seem to be no previous works dealing with asymptotic at infinity UCP. Motivated by such a question, we decided to study \eqref{gbo} with negative dispersions $-\frac{5}{2}<a<-2$, and we obtain some new UCP in Propositions \ref{propostimes1} and \ref{propostimes3} below.

Concerning conserved quantities, solutions of \eqref{gbo} formally preserve the mass 
\begin{equation*}
M[u(t)]=\int_{\mathbb{R}} |u(x,t)|^2\, dx,    
\end{equation*}
and the energy
\begin{equation}\label{energy}
E[u(t)]=\frac{1}{2}\int_{\mathbb{R}} |D^{\frac{a+1}{2}}u(x,t)|^2\, dx-\int_{\mathbb{R}} u^3(x,t)\, dx.    
\end{equation}
Motivated by the energy invariant \eqref{energy}, for $\theta>0$, $a<-2$, $s\in \mathbb{R}$, we consider the space
\begin{equation}\label{Hspacedefi}
H_{a,\theta}^s(\mathbb{R}):=H^s(\R)\cap \dot{H}^{(a+1)\theta}(\mathbb{R}),
\end{equation}  
where we recall that the homogeneous Sobolev space $\dot{H}^{(a+1)\theta}(\mathbb{R})$ consists of all tempered distributions $f$ for which $\widehat{f}\in L^1_{loc}(\mathbb{R})$ and $\|D^{(a+1)\theta}f\|=\||\xi|^{(a+1)\theta}\widehat{f}\|<\infty$, where $\|\cdot\|=\|\cdot\|_{L^2}$ denotes the usual norm in $L^2(\mathbb{R})$. Thus, we assign the norm $\|f\|_{H_{a,\theta}^s}^2:=\|f\|_{H^s}^2+\|D^{(a+1)\theta}f\|^2$ to the space $H_{a,\theta}^s(\mathbb{R})$. In particular, when $\theta=\frac{1}{2}$, the energy \eqref{energy} is justified in the space $H_{a,\frac{1}{2}}^s(\mathbb{R})$, $s>\frac{1}{2}$. 

Note that when $-\frac{5}{2}<\theta<a$, and $\frac{2+a}{1+a}\leq \theta<-\frac{3}{2(1+a)}$, a parabolic regularization-type argument can be used to show the existence of solutions to the Cauchy problem \eqref{gbo} in the space $H^{s}_{a,\theta}(\mathbb{R})$ for regularity $s>\frac{3}{2}$ (see Lemma \ref{existence} below). Such methods are standard and have been applied in quite a few contexts, see for example \cite{AbBonaFellSaut1989,APlow,Iorio1986,IorioNunes1998}. However, in the case of negative dispersion as that given in \eqref{gbo}, it has not been established whether such a method is valid for obtaining solutions in spaces $H^{s}_{a,\theta}(\mathbb{R})$. For this reason, we have decided to deduce the existence of solutions of \eqref{gbo} in Section \ref{standarlwpAppendix}. It must be highlighted, we do not know if there is a local well-posedness theory for \eqref{gbo} in spaces $H_{a,\theta}^s(\mathbb{R})$ for $0<\theta< \frac{2+a}{1+a}$ or $\theta\geq -\frac{3}{2(1+a)}$, and some $s\in \mathbb{R}$. Notice that in the case of negative dispersion $a< -2$, it is not clear if dispersive effects would help deduce the existence of solutions in spaces of lower regularity $s\leq \frac{3}{2}$.  

Next, we introduce our weighted Sobolev spaces. Given $a<-2$, $\theta>0$, $s,r\in \mathbb{R}$, we consider the space
\begin{equation}\label{defiZspace}
  Z_{s,r}^{a,\theta}=H^s_{a,\theta}(\R)\cap L^2(|x|^{2r}\, dx)
\end{equation}
endowed with the norm $\|f\|^2_{Z_{s,r}^{a,\theta}}=\|f\|^2_{H^s_{a,\theta}}+\||x|^{r}f\|^2$.  We also consider the space
\begin{equation}\label{defiZspace2}
  \mathcal{Z}_{s,r}^{a,\theta}=\{f\in  Z_{s,r}^{a,\theta}: |x|^{\theta-1}D^{1+a}f, D^{(1+a)(\theta-1)}(xf)\in L^2(\mathbb{R})\},
\end{equation}
equipped with the norm $\|f\|^2_{\mathcal{Z}_{s,r}^{a,\theta}}=\|f\|^2_{Z_{s,r}^{a,\theta}}+\||x|^{\theta-1}D^{1+a}f\|^2+ \|D^{(1+a)(\theta-1)}(xf)\|^2$. 

Our first main result focuses on studying well-posedness and persistence in the spaces $Z_{s,\theta}^{a,\theta}$, $\mathcal{Z}_{s,r}^{a,\theta}$, $\theta>0$, $s>0$, and $a<-2$. It must be clear that in this work, we follow Kato's notion of well-posedness, which consists in existence, uniqueness, persistence property (if $\phi \in X$ for some functional space, then the corresponding solution describes a continuous curve in $X$, i.e., $u\in  C([0,T];X)$), and continuous dependence of the map data-solution.

\begin{theorem}\label{lwpw} 
Assume $-\frac{5}{2}<a<-2$. The following statements hold:
\begin{itemize}
    \item[(i)] If $0<\theta\leq 1$ and $s>\frac{3}{2}$, 
    then the Cauchy problem \eqref{gbo} is locally well-posed in $Z_{s,\theta}^{a,\widetilde{\theta}}$, where $\widetilde{\theta}=\frac{2+a}{1+a}$, if $0<\theta<\frac{2+a}{1+a}$,  and $\widetilde{\theta}=\theta$, if $\theta\geq \frac{2+a}{1+a}$.
      \item[(ii)] If $1<\theta<-\frac{3}{2(1+a)}$ and $s\geq \max\{\frac{3}{2}^{+},\frac{5\theta((1+a)(\theta-1)+1) }{3(2\theta-1)-\theta}\}$, then the Cauchy problem \eqref{gbo} is locally well-posed in $\mathcal{Z}_{s,\theta}^{a,\theta}$.

      \item[(iii)] If $-\frac{3}{2(1+a)}\leq \theta<\frac{1+2a}{2(1+a)}$, $s\geq \max\{-\frac{9+6a}{2(5+2a)},2^{+}\}$ and $\phi\in \mathcal{Z}_{s,\theta}^{a,\theta}$, then there exist a time $T>0$ and a unique solution $u$ of \eqref{gbo} in the class
      \begin{equation}\label{finalclass2}
         C([0,T];H^s_{a,\widetilde{\theta}}(\mathbb{R}))\cap L^{\infty}([0,T];L^2(|x|^{2\theta}\, dx )), 
      \end{equation}
      with $\widetilde{\theta}=\Big(-\frac{3}{2(1+a)}\Big)^{-}$. Above, we follow the standard notation $a^{\pm }=a\pm\epsilon$, $0<\epsilon\ll 1$.
\end{itemize}
\end{theorem}

\begin{remark}
	Before discussing some points in the proof of Theorem \eqref{lwpw}, let us give some remarks.
	
(a) The local well-posedness in (i) and (ii) depend on the local existence of solutions in $H^s_{a,b}(\mathbb{R})$ given in Lemma \ref{existence}, where it is assumed that $b$ satisfies $\frac{2+a}{1+a}\leq b <-\frac{3}{2(1+a)}$.

(b) Contrary to (i) and (ii), under the assumptions in (iii), we cannot deduce the local well-posedness in $\mathcal{Z}_{s,\theta}^{a,\theta}$ because we do not know the local well-posedness in $H^s_{a,\theta}(\mathbb{R})$. The best we can say is the persistence result \eqref{finalclass2}. However, our proof extend to local well-posedness in $\mathcal{Z}_{s,\theta}^{a,\theta}$  provided that there exists a local theory in $H^s_{a,\theta}(\mathbb{R})$.
\end{remark}

Concerning the proof of Theorem \ref{lwpw}, given the difficulties in controlling the negative derivatives in weighted spaces, we first consider a regularized version of equation \eqref{gbo} (see \eqref{mugbo} below), for which a contraction argument shows persistence in weighted spaces. Afterward, the results for the Cauchy problem \eqref{gbo} follow by an approximation of such regularized equation. When $-\frac{3}{2(1+a)}\leq \theta<\frac{1+2a}{2(1+a)}$, we use the result of Theorem \ref{lwpw} for lower weights $\theta<-\frac{3}{2(1+a)}$, and an iterative argument to get the desired result. We remark that such an approach is possible by Proposition \ref{extradecaycond} that shows extra decay properties for the nonlinear quadratic term $u\partial_x u$. We also note that because of the negative dispersion in the equation, we are forced to use different arguments than those given for dispersion $a\geq -2$ in \cite{FLP1,fp,Rianogbo}. Nevertheless, we believe that our arguments can also be extended to such dispersion.

We emphasize that in the case of dispersion $a\geq -2$, $a\neq -1$, the number $\frac{7}{2}+a$ determines the maximum polynomial decay propagated by solutions of \eqref{gbo}. Contrary to that, when $-\frac{5}{2}<a<-2$, certain estimates that are only valid for such constraints (see \eqref{decompositionApart2}, and Proposition \ref{nonintegrabprop}), allow us to extend such a decay limit. Thus, in Theorem \ref{lwpw}, we obtain persistence results for weights $\theta<\frac{1+2a}{2(1+a)}$, where one has $\frac{1+2a}{2(1+a)}>\frac{7}{2}+a$. The above is one of the main differences of the fKdV equation when compared with $a<-2$.

Note that Theorem \ref{lwpw} establishes local well-posedness in spaces $Z_{s,r}^{a,\theta}$, $\mathcal{Z}_{s,r}^{a,\theta}$, where $r=\theta$ if $\theta\geq\frac{2+a}{1+a}$. Concerning cases where $r<\theta$, given $\frac{2+a}{1+a}\leq \widetilde{\theta}<-\frac{3}{2(1+a)}$, and $0<\theta\leq \widetilde{\theta}$, as a direct consequence of the proof of Theorem \ref{lwpw} (i) and (ii),  and the embedding $(H^s(\mathbb{R})\cap \dot{H}^{(1+a)\widetilde{\theta}}(\mathbb{R}))\hookrightarrow (H^s(\mathbb{R})\cap \dot{H}^{(1+a)\theta}(\mathbb{R}))$, $s>\frac{3}{2}$, we deduce:
\begin{corollary}\label{clwpw} 
 Assume $-\frac{5}{2}<a<-2$. If $\frac{2+a}{1+a}\leq \widetilde{\theta}\leq 1$, $0<\theta\leq \widetilde{\theta}$, and $s>\frac{3}{2}$, 
    then the Cauchy problem \eqref{gbo} is locally well-posed in $Z_{s,\theta}^{a,\widetilde{\theta}}$. If $1< \widetilde{\theta}< -\frac{3}{2(1+a)}$, $1<\theta\leq \widetilde{\theta}$, and $s\geq \max\{\frac{3}{2}^{+},\frac{5\theta((1+a)(\theta-1)+1) }{3(2\theta-1)-\theta}\}$, 
    then the Cauchy problem \eqref{gbo} is locally well-posed in $\mathcal{Z}_{s,\theta}^{a,\widetilde{\theta}}$.
\end{corollary}
In particular, setting $\widetilde{\theta}=\frac{1}{2}$, the previous corollary establish local well-posedness results for the space $H^s_{a,\frac{1}{2}}(\mathbb{R})\cap L^2(|x|^{2\theta}\, dx)$, $0<\theta\leq \frac{1}{2}$, which is motivated by the energy \eqref{energy}.

The conditions $D^{(1+a)\theta}\phi, |x|^{\theta-1}D^{1+a}\phi,D^{(1+a)(\theta-1)}(x\phi) \in L^2(\mathbb{R})$ in the assumptions of Theorem \ref{lwpw} may be completely technical. However, in the following theorem, we show that such conditions are necessary to propagate fractional weights in $L^2$-spaces. It seems these conditions appear to originate from the dispersive effects in fKdV. More precisely, we will show that if a sufficiently regular solution $u\in C([0,T];H^s(\mathbb{R}))$ of \eqref{gbo} with $u(0)=\phi$ persists in some space $L^{2}(|x|^{2\theta}\, dx)$ with $0\leq \theta\leq 1$, i.e., $u\in L^{\infty}([0,T];L^2(|x|^{2\theta}\,dx))$, then  $D^{(1+a)\theta}\phi, |x|^{\theta}\phi \in L^2(\mathbb{R})$ must hold. When $1<\theta<\frac{1+2a}{2(1+a)}$, we show that $u\in L^{\infty}([0,T];L^2(|x|^{2\theta}\,dx))$ implies that $D^{(1+a)(\theta-1)}(x\phi)$, $|x|^{\theta}\phi  \in L^2(\mathbb{R})$, if and only if $D^{(1+a)\theta}\phi$, $|x|^{(\theta-1)} D^{1+a}\phi\in L^2(\mathbb{R})$. Such a result connects the extra conditions in the definition of the space $\mathcal{Z}_{s,r}^{a,\theta}$ and the fact that $\phi \in Z_{s,r}^{a,\theta}$.

\begin{theorem}\label{Uniquecont1}
Suppose $-\frac{5}{2}<a<-2$ and $0<\theta<\frac{1+2a}{2(1+a)}$. Assume also
\begin{itemize}
	\item[(a)] $s>\frac{3}{2}$ if $0<\theta\leq 1$;
	\item[(b)]  $s\geq \max\{\frac{3}{2}^{+},\frac{5\theta((1+a)(\theta-1)+1) }{3(2\theta-1)-\theta}\}$ if $1<\theta<-\frac{3}{2(1+a)}$; and
	\item[(c)] $s\geq \max\{-\frac{9+6a}{2(5+2a)},2^{+}\}$ if $-\frac{3}{2(1+a)}\leq \theta<\frac{1+2a}{2(1+a)}$.
\end{itemize} 
 Let $u\in C([0,T];H^s(\mathbb{R}))\cap L^{\infty}([0,T];L^2(|x|^{2\theta}\, dx))$ be a solution of \eqref{gbo} with initial condition $\phi$. 
\begin{itemize}
    \item[(i)] If $0<\theta\leq 1$, then it must be the case  that
    \begin{equation*}
      D^{(1+a)\theta}\phi, \, \, |x|^{\theta}\phi \in L^2(\mathbb{R}).  
    \end{equation*}
    In particular, it follows that $u\in L^{\infty}([0,T];\dot{H}^{(1+a)\theta}(\mathbb{R}))$. 
\item[(ii)] Let $1<\theta<\frac{1+2a}{2(1+a)}$. Consider the following conditions on  $\phi$
\begin{equation}\label{uniquecontassump1}
  D^{(1+a)(\theta-1)}(x\phi), \, \, |x|^{\theta}\phi  \in L^2(\mathbb{R}),
\end{equation}
and
\begin{equation}\label{uniquecontassump2}
   D^{(1+a)\theta}\phi,\, \, |x|^{(\theta-1)} D^{1+a}\phi\in L^2(\mathbb{R}).
\end{equation}
Then  \eqref{uniquecontassump1} holds true if and only if \eqref{uniquecontassump2} holds true.
\end{itemize}
\end{theorem}

Note that in Theorem \ref{Uniquecont1}, we consider $u$ as a solution of \eqref{gbo} in the class $ C([0,T];H^s(\mathbb{R}))\cap L^{\infty}([0,T];L^2(|x|^{2\theta}\, dx))$, which in general does not imply that $|x|^{\theta}u(x,0) =|x|^{\theta}\phi(x) \in L^2(\mathbb{R})$. However, our UCP in Theorem \ref{Uniquecont1} actually shows that such a condition on the initial data must hold true. Moreover, in Theorem \ref{lwpw}, the propagation of weights requires a condition such as $u(t)\in \dot{H}^{(1+a)\widetilde{\theta}}(\mathbb{R})$ for some $\widetilde{\theta}>0$, but in contrast, the hypothesis of Theorem \ref{Uniquecont1} does not make use of such requirement. In this sense, Theorem \ref{Uniquecont1} applies to a wider class of solutions of \eqref{gbo}.

Let us deduce some consequences of Theorem \ref{Uniquecont1}. Consider  $\phi\in H^{\frac{3}{2}^{+}}(\mathbb{R})\cap L^2(|x|^{2\theta}\,dx)$ for some fixed $0<\theta\leq 1$, such that $D^{(1+a)\theta}\phi\notin L^2(\mathbb{R})$. Then, Theorem \ref{Uniquecont1} (i) shows that there does not exist a time $T_1>0$ and a solution $u$ of \eqref{gbo} with initial condition $\phi$ such that $u \in C([0,T_1];H^s(\mathbb{R}))\cap L^{\infty}([0,T_1];L^{2}(|x|^{2\theta}\, dx))$.  In this sense, the conditions $D^{(1+a)\theta}\phi, |x|^{\theta}\phi\in L^2(\mathbb{R})$ in Theorem \ref{lwpw} (i) are sharp.

When $\theta>1$, Theorem \ref{Uniquecont1} provides similar conclusions to those presented above. For example, let  $1<\theta<\frac{1+2a}{2(1+a)}$, $s>0$ be given as in Theorem \ref{Uniquecont1}, $\phi \in H^{s}(\mathbb{R})\cap L^2(|x|^{2\theta}\, dx)$ be such that \eqref{uniquecontassump1} holds true, and $D^{(1+a)\theta}\phi\notin L^2(\mathbb{R})$. For an example of such an initial condition, take $\varphi \in C^{\infty}_0(\mathbb{R})$ be such that $\varphi(\xi)=1$, whenever $|\xi|\leq 1$, and set $\phi=\varphi^{\vee}$. Thus, Theorem \ref{Uniquecont1} establishes that there does not exist a time $T_1>0$, and a solution $u$ of \eqref{gbo} with initial condition $\phi$ such that $u \in C([0,T_1];H^s(\mathbb{R}))\cap L^{\infty}([0,T_1];L^{2}(|x|^{2\theta}\, dx))$.

As a further consequence of the proof of Theorem \ref{Uniquecont1}, we obtain the following UCP for solutions of \eqref{gbo}.

\begin{proposition}\label{propostimes1}
Assume $-\frac{5}{2}<a<-2$ and $s\geq\max\{\frac{3}{(5+2a)},2^{+}\}$. Let $u\in C([0,T]; Z_{s,1}^{a,1})$ be a solution of \eqref{gbo}. If there exist two different times $t_1,t_2\in [0,T]$ such that
    \begin{equation*}
        u(t_j)\in L^2(|x|^{2^{+}}\, dx), \,\, j=1,2,
    \end{equation*}
then
\begin{equation}\label{identitywithtimes}
|x|^{0^{+}}U(t_j)((x\phi)-(2+a)t_jD^{1+a}\phi)\in L^2(\mathbb{R}), \, \, j=1,2. 
\end{equation}
In particular, the initial condition $\phi$ satisfies \eqref{uniquecontassump1} if and only if it satisfies \eqref{uniquecontassump2}.
\end{proposition}

\begin{proposition}\label{propostimes3}
Suppose $-\frac{5}{2}<a<-2$, $s\geq \max\{-\frac{9+6a}{2(5+2a)},2^{+}\}$, and $\theta=\Big(\frac{1+2a}{2(1+a)}\Big)^{-}$. Let $u\in C([0,T]; H^s_{a,\widetilde{\theta}}(\mathbb{R}))\cap L^{\infty}([0,T];L^2(|x|^{2\theta}\, dx))$ be a solution of \eqref{gbo} with $\widetilde{\theta}=\Big(-\frac{3}{2(1+a)}\Big)^{-}$. If there exist two times $t_1,t_2\in [0,T]$ with $t_2>t_1$ such that
    \begin{equation*}
        u(t_j)\in L^2(|x|^{2\big(\frac{1+2a}{2(1+a)}\big)}\, dx), \,\, j=1,2,
    \end{equation*}
    and
 \begin{equation}\label{extracondtime1}
D^{(1+a)\big(\frac{1+2a}{2(1+a)}\big)}u(t_1),\, \, D^{(1+a)\big(\frac{1+2a}{2(1+a)}-1\big)}(xu(t_1)),\, \, |x|^{\frac{1+2a}{2(1+a)}-1}D^{1+a}u(t_1)\in L^2(\mathbb{R}),
    \end{equation}
then
\begin{equation*}
u\equiv 0.
\end{equation*}
\end{proposition}

We remark that the result in Proposition \ref{propostimes3} establishes that $\frac{1+2a}{2(1+a)}$ is the maximum polynomial decay propagated by solutions of \eqref{gbo}. Indeed, let $-\frac{5}{2}<a<-2$, $s>0$ as in Proposition \ref{propostimes3}, and consider $\phi\in H^s(\mathbb{R})\cap L^2(|x|^{2\big(\frac{1+2a}{2(1+a)}\big)}\, dx)$ such that
 \begin{equation*}
D^{(1+a)\big(\frac{1+2a}{2(1+a)}\big)}\phi,\, \, D^{(1+a)\big(\frac{1+2a}{2(1+a)}-1\big)}(x \phi),\, \, |x|^{\frac{1+2a}{2(1+a)}-1}D^{1+a}\phi\in L^2(\mathbb{R}).        
\end{equation*}
Notice that such conditions are motivated by the results of Theorems \ref{lwpw} and Proposition \ref{propostimes1}. Then if $\phi \neq 0$, Theorem \ref{lwpw} assures that there exist a time $T>0$ and a unique solution $u$ of \eqref{gbo} in the class $C([0,T]; H^s_{a,\widetilde{\theta}}(\mathbb{R}))\cap L^{\infty}([0,T];L^2(|x|^{2\theta}\, dx))$ with $\theta=\Big(\frac{1+2a}{2(1+a)}\Big)^{-}$ and $\widetilde{\theta}=\Big(-\frac{3}{2(1+a)}\Big)^{-}$. However, Proposition \ref{propostimes3} implies that for any time $T_1>0$, $u \notin L^{\infty}([0,T_1];L^2(|x|^{2\big(\frac{1+2a}{2(1+a)}\big)}\, dx))$. 

\begin{remark}
Let us compare our results with the conclusions for $a\geq -2$ and $a\neq -1$ in \eqref{gbo} obtained in \cite{fp,Rianogbo,FLP1}. 

(a) In the case of Theorem \ref{lwpw}, the conditions $D^{(1+a)\theta}\phi, |x|^{\theta-1}D^{1+a}\phi, D^{(1+a)(\theta-1)}(x\phi) \in L^2(\mathbb{R})$ for dispersions $a\geq -2$ are not necessary and are obtained in different ways. For example, it is common to use the zero-mean property $\widehat{\phi}(0)=0$ to deduce that $D^{(1+a)\theta}\phi \in L^2(\mathbb{R})$. For this reason, when $a\geq -2$  one uses the spaces $H^s(\mathbb{R})\cap L^2(|x|^{2\theta}\, dx)\cap \{ f\in H^s(\mathbb{R}): \widehat{f}(0)=0\}$. However, when $a<-2$ the derivative $D^{(1+a)\theta}$ is too singular to guarantee integrability for arbitrary functions with zero mean. Also, note that if $\phi\in L^2(|x|^{2\theta}\, dx)$, then $\widehat{\phi}\in H^{\theta}(\mathbb{R})$, but in our case $\theta<\frac{1+2a}{2(1+a)}<\frac{3}{2}$, which is small to extract useful properties for $\widehat{\phi}$ and $\partial_{\xi}\widehat{\phi}$, as to say that it is H\"older or Lipschitz continuous. Hence, in this work, we have substituted the zero mean property by hypotheses $D^{(1+a)\theta}\phi, |x|^{\theta-1}D^{1+a}\phi, D^{(1+a)(\theta-1)}(x\phi) \in L^2(\mathbb{R})$.    

(b) In the case of Proposition \ref{propostimes1}, when $a>-2$, $a\neq -1$, in \cite{fp,Rianogbo,FLP1}, it has been proved that if for two different times $t_1,t_2$ one has that a solution $u$ of \eqref{gbo} satisfies $|x|^{\frac{5}{2}+a}u(t_j)\in L^2(\mathbb{R})$, then the zero mean holds true, i.e., $\widehat{\phi}(0)=0$. However, when $a<-2$, as we mentioned before, the condition $\widehat{\phi}(0)=0$ is being replaced by \eqref{identitywithtimes}. We emphasize that when $a \geq -2$, using \eqref{identitywithtimes} an a few extra steps, we obtain the UCP results deduced in \cite{fp,Rianogbo,FLP1}.

(c) In the case of the UCP in Proposition \ref{propostimes3}, when the dispersion $a>-2$ with $a\neq -1$, in \cite{fp,Rianogbo,FLP1}, it has been proved that if for three different times $t_1,t_2,t_3$, the solution of \eqref{gbo} satisfies $|x|^{\frac{7}{2}+a}u(t_j)\in L^2(\mathbb{R})$, then $u=0$. Such a result depends strongly on the following identity for the solution of \eqref{gbo},
\begin{equation}\label{identitypositivea}
\int xu(x,t)\, dx =\int x \phi(x)\, dx+\frac{t}{2}\|\phi\|^2.
\end{equation}
In our case, since we study arbitrary initial conditions for which the maximum decay limit is $\frac{1+2a}{2(1+a)}<\frac{3}{2}$,  the above identity is not always justified.  Additionally, for $a<-2$, we note that estimates \eqref{decompositionApart} and \eqref{decompositionApart2} are more beneficial than in the case $a>-2$.  For this reason, we do not need identity \eqref{identitypositivea}, and we can extend the polynomial decay propagated by solutions of \eqref{gbo} to $\frac{1+2a}{2(1+a)}> \frac{7}{2}+a$. Complementing this, our results from Proposition \ref{propostimes3} show that if a sufficiently regular solution of \eqref{gbo} decays at two times as $\frac{1+2a}{2(1+a)}$ and \eqref{extracondtime1} holds, then such a solution must be identically zero. Note that although we do not use three times (which is the typical assumption when $a>-2$), we do have two times conditions and $|x|^{\theta-1}D^{1+a}\phi, D^{(1+a)(\theta-1)}(x\phi) \in L^2(\mathbb{R})$, which somehow shows that even for the case $a<-2$, three conditions are still needed to get Proposition \ref{propostimes3}.
\end{remark}

This paper is organized as follows. In Section \ref{notation}, we present the notation and preliminaries needed to derive our results. In Section \ref{localwellp}, we show our well-posedness and persistence results in weighted spaces stated in Theorem \ref{lwpw}. In Section \ref{uniquesect}, we deduce our unique continuation principles. More precisely, we prove Theorem \ref{Uniquecont1}, and Propositions \ref{propostimes1} and \ref{propostimes3}. Finally, we conclude with an appendix where we show local well-posedness result for the Cauchy problem \eqref{gbo} in spaces $H^s_{a,\theta}(\mathbb{R})$.


\section{Notation and Preliminaries}\label{notation}

We will use standard notation. Given $a$ and $b$ be two positive numbers, we say that $a\lesssim b$ if there exists a constant $c>0$ such that $a\leq c b$, and we say that $a \gtrsim b$ if $b\lesssim a$. We write $a\sim b$ if $a\lesssim b$ and $b\lesssim a$. The commutator between two operators $A$, $B$ is denoted as $[A,B]=AB-BA$. 

$C^{\infty}_0(\mathbb{R})$ denotes the class of smooth functions with compact support on $\mathbb{R}$. Given $1\leq p\leq \infty$,  we denote by $\|\cdot\|_{L^p}$ the usual norm in $L^p(\mathbb{R})$; when $p=2$, as we mentioned before, we write $\|\cdot\|_{L^2}=:\|\cdot\|$, but depending on the context and to emphasize the variable of integration, we will also use $\|\cdot\|_{L^2}$. Given $s\in \mathbb{R}$, $J^{s}$ is defined via Fourier transform as $\widehat{J^{s} f}(\xi)=\langle \xi \rangle^{s} \widehat{f}(\xi)$, with $\langle \xi \rangle=(1+|\xi|^2)^{1/2}$. Thus, the Sobolev spaces $H^s(\mathbb{R})$ consists of all tempered distributions $f$ such that $\|f\|_{H^s}=\|J^sf\|<\infty$.

The group of bounded linear operators associated with the linear equation $\partial_tu-\partial_x D^{a+1}u=0$ is defined via Fourier transform by
\begin{equation}\label{group}
(U(t)\phi)^{\wedge}(\xi)=e^{it\xi |\xi|^{1+a}}\widehat{\phi} (\xi), \quad t\in \R.
\end{equation}

The following fractional Leibniz rule (see \cite{Grafakos,KPV}) will be used in our arguments.

\begin{lemma}
Let $s>0$ and $1<p<\infty$, then
\begin{equation}\label{leibfract}
\|D^s(hf)\|_{L^p}\les \|D^s h\|_{L^p}\|f\|_{L^\infty}+\|D^s f\|_{L^p}\|h\|_{L^\infty}.
\end{equation}

\end{lemma}
We will apply the following commutator estimate.
\begin{proposition}\label{comuDs1}
Let $s\in (0,1)$ and $1<p<\infty$, then 
\begin{equation}\label{comuDs2}
\|[D^s,g]f\|_{L^p} \les \|D^s g\|_{L^\infty}\|f\|_{L^p},
\end{equation}
and
\begin{equation}\label{comuDs3}
\|[D^s,g]f\|_{L^p} \les \|D^s g\|_{L^p}\|f\|_{L^{\infty}}.
\end{equation}
\end{proposition}
\begin{proof}
See \cite{dong} and reference therein.
\end{proof}

Next, we introduce equivalent definitions for fractional derivatives. Although we work in one spatial dimension, we will state the following results for any dimension $d\geq 1$. We denote by $L^{p}_{b}(\mathbb{R}^d)$ the fractional Sobolev space defined as $L^{p}_{b}(\mathbb{R}^d):=(1-\Delta)^{-\frac{b}{2}}L^{p}(\R^d)$. Such spaces can be characterized by the Stein derivative of order $b$.

\begin{theorem}\label{stein}
	Let $b\in (0,1)$ and $\frac{2d}{d+2b}<p<\infty.$ Then $f\in L^{p}_{b}(\R^{d})$ if and only if
	\begin{itemize}
		\item [a)] $f\in L^{p}(\R^{d}),$ 
		\item [b)]
		$\mathcal{D}^{b}f(x):={\displaystyle \left (
			\int_{\R^{d}}\frac{|f(x)-f(y)|^{2}}{|x-y|^{d+2b}}dy\right)^{\frac{1}{2}}} \in
		L^{p}(\R^{d}),$ with
		\begin{equation*}
		\|f\|_{b,p}:=\|J^{b}f\|_{p}\sim \|f\|_{p}+\|D^{b}f\|_{p}\sim \|f\|_{p}+\|\mathcal{D}^{b}f\|_{p}.
		\end{equation*}
  In particular, when $p=2$, $\|D^b f\|\sim \|\mathcal{D}^bf\|$.
	\end{itemize}
\end{theorem}
\begin{proof}
	We refer to \cite{stein}.
\end{proof}

One of the advantages of using $\mathcal{D}^{b}$ is that it provides a suitable formula to obtain pointwise estimates for fractional derivatives. In addition, from  Fubini's theorem, we have the following product estimate 
\begin{equation}\label{Leib}
\|\mathcal{D}^{b}(fg)\|_{L^2(\R^d)} \leq \|f\mathcal{D}^{b}g\|_{L^2(\R^d)} + \|g\mathcal{D}^{b}f\|_{L^2(\R^d)}.
\end{equation}
As a particular case of the previous inequality, and the definition of the derivative $\mathcal{D}^b$, we have:

\begin{proposition}\label{Leibnitz}
Let $b\in (0,1)$ and $h$ be a measurable function on $\R^d$ such that $h,\nabla h\in L^{\infty}(\R^d)$. Then, for almost every $x\in \R^d$,
\begin{equation}\label{Lei}
\mathcal{D}^b h(x)\lesssim \|h\|_{L^{\infty}(\R^d)}+\|\nabla h\|_{L^\infty(\R^d)}.
\end{equation}
Moreover,
\begin{equation}\label{Leibh}
\|\mathcal{D}^{b}(h f)\|_{L^2(\R^d)} \leq \|\mathcal {D}^b h\|_{L^\infty(\R^d)} \|f\|_{L^2(\R^d)} + \|h\|_{L^\infty(\R^d)} \|\mathcal{D}^{b}f\|_{L^2(\R^d)}.
\end{equation}
\end{proposition}

We will also use the following interpolation estimate.
\begin{lemma}\label{interx}
Let $\alpha,b>0.$ Assume that $J^{\alpha}f\in L^{2}(\R^d)$ and
$\langle x \rangle^b f=(1+x^2)^{\frac{b}{2}}f(x)\in L^{2}(\R^d).$ Then, for any
$\beta \in (0,1)$,
\begin{equation}\label{inter1x}
\|J^{\alpha \beta}(\langle x \rangle^{(1-\beta)b}f)\|_{L^2(\mathbb{R}^d)}\les \|\langle  x
\rangle^{b}f\|_{L^2(\mathbb{R}^d)}^{1-\beta}\|J^{\alpha}f\|_{L^2(\mathbb{R}^d)}^{\beta}.
\end{equation}
\end{lemma}
\begin{proof}
We refer to \cite[Lemma 4]{NahasPonce}. 
\end{proof}


\subsection{Weighted and fractional derivative estimates}

In this part,  we deduce different estimates using the fractional derivative $\mathcal{D}^{b}$ introduced earlier. The goal is to deduce estimates in weighted spaces for the group $\{U(t)\}$. For this, we first obtain fractional derivative bounds for the multiplier associated with the family $\{U(t)\}$.

\begin{proposition}\label{pontualn}
Let $a\in (-\frac52,-2)$ and $b\in (0,1)$. Then for any $t\in \R$ it follows that
\begin{equation*}
\mathcal D^b(e^{itx |x|^{1+a}})\les \langle t \rangle |x|^{(1+a)b},
\end{equation*}
for each  $x\neq 0$. The implicit constant above is independent of $x$ and $t$. 
\end{proposition}

\begin{proof}

Setting $A=\{y\in \R: |y-x|\geq  |x|^{-1-a}\}$, we first divide our estimates as follows
\begin{equation*}
\begin{split}
\left(\mathcal D^b(e^{itx |x|^{1+a}})\right)^2=\int \dfrac{\big|e^{itx |x|^{1+a}}-e^{ity |y|^{1+a}}\big|^2}{|y-x|^{1+2b}}dy=\int_{A}(\cdots)+\int_{\R\setminus A}(\cdots) =:I_1+I_2.
\end{split}
\end{equation*}
A change of variable and the definition of $A$ yield
\begin{equation*}
I_1\les \int_{|x|^{-1-a}}^\infty r^{-1-2b}dr\sim |x|^{2(1+a) b}.
\end{equation*}
To deal with $I_2$, we introduce the sets $B=\{y\in \R\setminus A: \;|x|<|y|\}$ and $C=\{y\in \R\setminus A:\;|y|\leq |x|\}$ to write
\begin{equation*}
\begin{split}
I_2&=\int_{\R\setminus A}\dfrac{\big|1-e^{it(x |x|^{1+a}-y |y|^{1+a})}\big|^2}{|y-x|^{1+2b}}dy=\int_B (\cdots)+\int_C (\cdots)=:I_3+I_4.
\end{split}
\end{equation*}
Now, since 
\begin{equation}\label{estimaxy}
|x |x|^{1+a}-y |y|^{1+a}|\leq |y-x||x|^{1+a}+|y|\big||x|^{1+a}-|y|^{1+a}\big|,
\end{equation}
we get
\begin{equation*}
\begin{split}
I_3
&\leq |t|^2\int_{B} \dfrac{\big|x |x|^{1+a}-y |y|^{1+a}\big|^2}{|y-x|^{1+2b}}dy\\
&\leq |t|^2 \Bigg(|x|^{2(1+a)}\int_{B} |y-x|^{1-2b}dy+\int_{B} \dfrac{|y|^2\big||x|^{1+a}-|y|^{1+a}\big|^2}{|y-x|^{1+2b}}dy\Bigg)=:I_{3,1}+I_{3,2}.
\end{split}
\end{equation*}
Using that $B\subset \mathbb{R}\setminus A$, the first term on the right-hand side of the inequality above can be estimated in the following manner
\begin{equation}\label{I31}
\begin{split}
I_{3,1} &\leq |t|^2 |x|^{2(1+a)} \int_{|y-x|< |x|^{-1-a}}|y-x|^{1-2b}dy \\
& \les |t|^2 |x|^{2(1+a)}(|x|^{-1-a})^{2-2b}\\
&\les |t|^2 |x|^{2(1+a)b}.
\end{split}
\end{equation}
To deal with $I_{3,2}$, using that $0<|x|<|y|$ for all $y\in B$, we can apply the mean value theorem to find
\begin{equation}\label{meanvalue}
||x|^{1+a}-|y|^{1+a}|\lesssim |x|^{a}|y-x|,
\end{equation}
this inequality together with the fact that $|y|^{2(1+a)}\leq |x|^{2(1+a)}$ allow us to infer
\begin{equation}\label{intB}
\begin{split}
I_{3,2}&\lesssim  |t|^2\int_{B} |y-x|^{1-2b}\big||x|^{1+a}-|y|^{1+a}\big|^2 dy+|t|^2\int_{B} \dfrac{|x|^2\big||x|^{1+a}-|y|^{1+a}\big|^2}{|y-x|^{1+2b}}dy\\
&\les |t|^2\int_B |y-x|^{1-2b}\big(|x|^{2(1+a)}+|y|^{2(1+a)}\big)dy+|t|^2|x|^{2(1+a)}\int_B |y-x|^{1-2b}dy\\
&\les |t|^2 |x|^{2(1+a)}\int_{_{|y-x|< |x|^{-1-a}}}|y-x|^{1-2b}dy\\
&\les |t|^2|x|^{2(1+a)b}.
\end{split}
\end{equation}
Next, we deal with $I_4$. From \eqref{estimaxy}, we have
\begin{equation*}
\begin{split}
I_4
&\leq |t|^2\int_{C} \dfrac{\big|x |x|^{1+a}-y |y|^{1+a}\big|^2}{|y-x|^{1+2b}}dy\\
&\leq |t|^2 \Bigg(|x|^{2(1+a)}\int_{C} |y-x|^{1-2b}dy+\int_{C} \dfrac{|y|^2\big||x|^{1+a}-|y|^{1+a}\big|^2}{|y-x|^{1+2b}}dy\Bigg)=:I_{4,1}+I_{4,2}.
\end{split}
\end{equation*}
In a similar way to \eqref{I31}, it follows that $I_{4,1}\les |t|^2 |x|^{2(1+a)b}$. On the other hand, given that $a<-2$, we use the mean value theorem to deduce
\begin{equation}\label{Cxy}
||x|^{1+a}-|y|^{1+a}|=|y|^{1+a}\big||x|^{-1-a}-|y|^{-1-a}\big||x|^{1+a}\les |y|^{1+a}|x|^{-1}|y-x|,
\end{equation}
for all $y\in C$. Hence, we get
\begin{equation*}
\begin{split}
I_{4,2}&\les |t|^2\int_C |y|^{2(2+a)}|x|^{-2}|y-x|^{1-2b}dy\\
& \les |t|^2\int_{C\cap\{\frac{|x|}{2}\leq |y|\}} (\cdots)\, dy+|t|^2\int_{C\cap\{|y|\leq \frac{|x|}{2}\}} (\cdots)\, dy\\
&=:I_{4,2,1}+I_{4,2,2}.
\end{split}
\end{equation*}
On the support of the integral defining  $I_{4,2,1}$, we have that $|y|\sim |x|$ and $|x-y|< |x|^{-1-a}$, so
\begin{equation*}
\begin{split}
I_{4,2,1}&\les |t|^2|x|^{2(1+a)}\int_{|x-y|< |x|^{-1-a}} |y-x|^{1-2b}dy\\
& \les |t|^2|x|^{2(1+a)b}.
\end{split}
\end{equation*}
Now, on the support of the integral defining  $I_{4,2,2}$,  we have that $\frac{|x|}{2}\leq |y-x|\leq \min\{\frac{3}{2}|x|,|x|^{-1-a}\}$, which forces us to have $|x|\leq 2|x|^{-1-a}$, i.e., $2^{\frac{1}{2+a}}\leq |x|$. Thus, using that $|y-x|\sim |x|$, we get
\begin{equation*}
\begin{split}
I_{4,2,2}&\les |t|^2|x|^{-1-2b}\int_{|y|\leq |x|} |y|^{2(2+a)}dy\\
& \les |t|^2|x|^{2(2+a)-2b}\\
&\les |t|^2|x|^{2(1+a)b},
\end{split}
\end{equation*}
where the integrability of $|y|^{2(2+a)}$ follows form the fact that $a>-\frac{5}{2}$, and in the last line we have used that $a<-2$ implies $2(2+a)-2b\leq 2(1+a)b$, and that $2^{\frac{1}{2+a}}\leq |x|$. This finishes the proof.
\end{proof}

Motivated by the regions used in the proof of Proposition \ref{pontualn}, we can prove the following result, which will be useful to study the cases $0<b<1$, where $U(t)\varphi \notin L^2(|x|^{2b}\, dx)$.

\begin{proposition}\label{nonintegrabprop}
    Assume $a\in (-\frac{5}{2},-2)$ and let $\varphi\in C^{\infty}_0(\mathbb{R})$ be such that $\varphi\equiv c$ for some constant $c\neq 0$ in some neighborhood of the origin. Consider $\frac{-1}{(1+a)}<p<\infty$, and $\frac{-1}{p(1+a)}\leq b<1$. Then
\begin{equation*}
\mathcal{D}^{b}\big(e^{itx|x|^{1+a}}\varphi \big) \notin L^p(\mathbb{R}), 
\end{equation*}
for any $ t\in \mathbb{R}\setminus\{0\}$.
\end{proposition}

\begin{proof} Without loss of generality, we will assume that $\varphi(x)=1$ on $|x|\leq 1$. Given $x\in (0,\frac{1}{2})$, we consider $\mathcal{B}_1(x)=\{y>0:|x-y|\leq \big(\min\{(100|t|)^{-1},\frac{1}{2}\}\big)|x|^{-1-a}\}$.  The definition of the operator $\mathcal{D}^{b}$ implies that for all $x\in (0,\frac{1}{2})$,
\begin{equation*}
\begin{aligned}
\Big(\mathcal{D}^{b}\big(e^{itx|x|^{1+a}}\varphi \big) \Big)^2(x)\geq \int_{\mathcal{B}_1(x)} \frac{|\sin\big(t(x|x|^{1+a}-y|y|^{1+a})\big)|^2}{|x-y|^{1+2b}}\, dy, 
\end{aligned}    
\end{equation*}
where we have used that $\varphi(x)=\varphi(y)=1$ when $x\in (0,\frac{1}{2})$, and $y\in \mathcal{B}_1(x)$. To complete the estimate of the previous inequality, let us obtain first some preliminary results. 

Given that $|x|^{-1-a}\leq |x|$ provided that $x\in (0,\frac{1}{2})$, and $a<-2$, we have for all $y\in \mathcal{B}_1(x)$ that
\begin{equation}\label{nonintegeq1}
    |y|\geq |x|-|x-y|\geq |x|-\frac{1}{2}|x|^{-1-a}\geq \frac{1}{2}|x|,
\end{equation}
and since $|y|\leq 2|x|$,  it follows $|x|\sim |y|$. Now, setting $F(x)=x|x|^{1+a}$, the mean value theorem assures the existence of some $z\in [|y|,|x|]$ such that 
\begin{equation*}
\begin{aligned}
|t(x|x|^{1+a}-y|y|^{1+a})|=&|t||F'(z)||x-y|\\
=&|2+a||t||z|^{1+a}|x-y|.
\end{aligned}    
\end{equation*}
Thus, it follows from \eqref{nonintegeq1} and the definition of $\mathcal{B}_1(x)$ that
\begin{equation}\label{nonintegeq2}
\begin{aligned}
 |t(x|x|^{1+a}-y|y|^{1+a})|\leq &  |2+a||t||y|^{1+a}|x-y|\\
\leq & 2^{-(1+a)}|2+a||t||x|^{1+a}\big((100|t|)^{-1}|x|^{-1-a}\big) \\
\leq& \frac{1}{2}.
\end{aligned}
\end{equation}
Additionally, one has
\begin{equation}\label{nonintegeq3}
    |t(x|x|^{1+a}-y|y|^{1+a})||\gtrsim |t||x|^{1+a}|x-y|,
\end{equation}
for all $y\in \mathcal{B}_1(x)$. It follows from \eqref{nonintegeq2} and \eqref{nonintegeq3} that
\begin{equation*}
\begin{aligned}
|\sin\big(t(x|x|^{1+a}-y|y|^{1+a})\big)|\gtrsim &  |t(x|x|^{1+a}-y|y|^{1+a})|  \\
\gtrsim & |t||x|^{1+a}|x-y|,
\end{aligned}
\end{equation*}
where the implicit constant is independent of $x\in (0,\frac{1}{2})$, and $y\in \mathcal{B}_1(x)$. Thus, going back to the first estimate, we get
\begin{equation*}
\begin{aligned}
\Big(\mathcal{D}^{b}\big(e^{itx|x|^{1+a}}\varphi \big) \Big)^2(x)\gtrsim & |t|^2|x|^{2(1+a)}\int_{\mathcal{B}_1(x)}|x-y|^{1-2b}\, dy\\
\gtrsim_{t}& |x|^{2(1+a)}(|x|^{-1-a})^{2-2b}\\
\sim_t & |x|^{2(1+a)b}.
\end{aligned}    
\end{equation*}
Since $b\geq \frac{-1}{p(1+a)}$ implies that $|x|^{(1+a)b} \notin L^p((0,\frac{1}{2}))$, the desired result is a consequence of the previous inequality.

\end{proof}

As a consequence of Proposition \ref{pontualn}, we deduce weighted estimates in $L^2$-Sobolev spaces for solutions of the linear equation $\partial_t u-\partial_x D^{a+1} u=0$.

\begin{lemma}\label{Dthetau}
 Let $a\in (-\frac52,-2)$ and $\theta \in (0,1]$. Then for all $t\in \mathbb{R}$,
\begin{equation*}
\||x|^\theta U(t)f\|\les \Big(\langle t \rangle\|D^{(1+a)\theta}f\|+\|\lanx^\theta f\|\Big),
\end{equation*}
where $U(t)$ is the linear group defined in \eqref{group}. Moreover, if $1<\theta\leq 2$, we have
\begin{equation*}
\begin{aligned}
\||x|^\theta U(t)f\|\les & \langle t\rangle^{2}\|D^{(1+a)\theta}f\|+\langle t\rangle\big\||x |^{\theta-1} D^{1+a}f\|\\
&+\langle t\rangle\big\|D^{(1+a)(\theta-1)}(xf)\|+\|\langle x \rangle^{\theta}f\|.
\end{aligned}    
\end{equation*}

\end{lemma}

\begin{proof}
When $0<\theta < 1$, changing to the frequency domain, the proof follows from \eqref{Leib} and Proposition \ref{pontualn}. The case $\theta=1$ follows from the fact that the weight $x$ in frequency domain transfers to a local derivative $\frac{\partial}{\partial \xi}$.  Now, when $\theta>1$, we write $\theta=1+\theta_1$. We observe
\begin{equation*}
\begin{aligned}     
\frac{\partial}{\partial \xi}\big(e^{it\xi|\xi|^{1+a}} \widehat{f}\big)=e^{it\xi|\xi|^{1+a}}\big((2+a)it|\xi|^{1+a}\widehat{f}+\frac{\partial}{\partial \xi}\widehat{f}\big).
\end{aligned}   
\end{equation*}
Using the previous identity and the case $0<\theta_1\leq 1$ above, we deduce
\begin{equation*}
\begin{aligned}
\||x|^\theta U(t)f\|\leq & \||x|^{\theta_1}U(t)\big((2+a)|t| D^{1+a}f-xf \big)\|  \\
\les & \langle t\rangle^{2}\|D^{(1+a)(1+\theta_1)}f\|+\langle t\rangle\big\|\langle x \rangle^{\theta_1}D^{1+a}f\|\\
&+\langle t\rangle\big\|D^{(1+a)\theta_1}(xf)\|+\|\langle x \rangle^{(1+\theta_1)}f\|.
\end{aligned}    
\end{equation*}
The proof is thus completed.
\end{proof}

We also consider the following semigroup  $\{U_{\mu}(t)\}$ associated with the initial-value problem  $\partial_t u-\partial_x D^{a+1}u=\mu\partial_x^2u$, $u(0)=\phi$, which is defined via Fourier transform as
\begin{equation}\label{musemi}
(U_\mu (t)\phi)^{\wedge}(\xi)=e^{it\xi |\xi|^{1+a}-\mu t |\xi|^2}\widehat{\phi} (\xi), \quad t\in [0,\infty), \quad \mu>0.
\end{equation}
The following result will be useful in Section \ref{localwellp}, i.e., in the proof of Theorem \ref{lwpw}.

\begin{corollary}\label{Dmu}
  Let $a\in (-\frac52,-2)$, and $\theta, \mu \in (0,1)$. Then  for each $t\in \mathbb{R}$,
\begin{equation*}
\||x|^\theta U_\mu(t)f\|\les \langle t \rangle\Big(\|D^{(1+a)\theta}f\|+\|\lanx^\theta f\|\Big),
\end{equation*}
where $U_\mu(t)$ is defined in \eqref{musemi}.  Moreover, if $1<\theta\leq 2$, we have
\begin{equation*}
\begin{aligned}
\||x|^\theta U_{\mu}(t)f\|\les & \langle t\rangle^{2}\Big(\|D^{(1+a)\theta}f\|+\big\||x|^{\theta-1}D^{1+a}f\|+\big\|D^{(1+a)(\theta-1)}(xf)\|+\|\langle x \rangle^{\theta}f\|\Big).
\end{aligned}
\end{equation*}
The implicit constants in the previous inequalities are independent of $\mu$.    
\end{corollary}

\begin{proof}
The proof follows from Lemma \ref{Dthetau}, and Proposition \ref{Leibnitz} making $h(\xi)$ equal to $e^{-\mu t |\xi|^2}$, and $-2\mu t \xi e^{-\mu t |\xi|^2}$.     
\end{proof}

Now, we show that the semigroup $\{U_{\mu}(t)\}$ regularizes the initial condition in the following sense. 
\begin{proposition}\label{regularity}
Let $\lambda\in [0,\infty)$, $\mu>0$ and $a<-2$. For any $s\in \mathbb{R}$, and $t>0$,
\begin{equation*}
\|U_\mu(t)\phi\|_{H^{s+\lambda}}\les\Big(1+(\mu t)^{-\frac{\lambda}{2}}\Big)\|\phi\|_{H^s}, \qquad \phi\in H^s(\R),
\end{equation*}
where the implicit constant depends on $\lambda$.
\end{proposition}

\begin{proof}
We first claim that for $\lambda,\mu,t>0$, it follows 
\begin{equation}\label{regulineq}
\|\xi^{2\lambda} e^{-\mu t\xi^2}\|_{L^{\infty}_{\xi}}\leq (\lambda^{-1} \mu t e)^{-\lambda}. 
\end{equation}
The previous inequality follows by setting $\omega=(t\mu)^{\frac{1}{2}}\xi$, and maximizing the function $\omega^{2\lambda}e^{-\omega^2}$, $\omega>0$. For a similar idea, see \cite[Lemma 2.1]{wfcr}. Now, since $\|U_\mu(t)\phi\|_{H^{s+\lambda}}=\|\langle \xi \rangle^{s+\lambda}e^{-\mu t\xi^2}\widehat{\phi}(\xi)\|$, the proof of Proposition \ref{regularity} is a consequence of \eqref{regulineq}. 
\end{proof}

\begin{proposition}\label{localdecaylemma}
Let $\varphi\in C^{\infty}_0(\mathbb{R})$. Then for any $\theta \in (0,1)$, and $0<\beta<\frac{1}{2}$, we have
\begin{equation}\label{Dstein2}
\mathcal{D}^\theta (|\xi|^{-\beta} \varphi (\xi))(x)\les |x|^{-\beta-\theta},
\end{equation}
for all $x\neq 0$. Similar conclusion holds for $\mathcal{D}^\theta (|\xi|^{-\beta}\sgn(\xi)\varphi(\xi))$.
\end{proposition}

\begin{proof}
We refer to \cite[Proposition 2.13]{Rianogbo}.    
\end{proof}


\section{Persistence results in weighted spaces}\label{localwellp}

To deduce Theorem \ref{lwpw}, we will use a regularized version of \eqref{gbo}. More precisely, let $\mu\in(0,1)$, and consider the following Cauchy problem
\begin{equation}\label{mugbo}
    \left\{\begin{aligned}
    &\partial_t u -\partial_{x}D^{1+a} u+u\partial_xu=\mu \p_x^2 u,\quad x\in \mathbb{R}, \, t\in \mathbb{R}, \hspace{0.5cm} a\in (-\frac52,-2), \\
    &u(x,0)=\phi(x).
    \end{aligned}\right.
    \end{equation}
Following similar arguments in \cite{AbBonaFellSaut1989,APlow,Iorio1986}, see also the work for the KP equation in \cite{IorioNunes1998}, which exemplifies a study of negative dispersion, one can use a standard parabolic regularization argument to take $\mu\to 0$ in  \eqref{mugbo} to deduce the following local well-posedness result in the space $H_{a,\theta}^s(\mathbb{R})$ (recall definition \eqref{Hspacedefi}).
\begin{lemma}\label{existence} 
Let $-\frac{5}{2}<a<-2$, $\frac{2+a}{1+a}\leq \theta<-\frac{3}{2(1+a)}$, and $s>\frac{3}{2}$ be fixed. Consider, $\phi\in H^s_{a,\theta}(\R)$. Then for each $0\leq \mu<1$, there exist a common time $T=T(\|\phi\|_{H^s_{a,\theta}})>0$ independent of $\mu$, and a unique solution $u_\mu$ of \eqref{mugbo} if $0<\mu<1$, and of \eqref{gbo} if $\mu=0$, with common initial condition $u_{\mu}(0)=\phi$ such that
\begin{equation*}
u_{\mu}\in C([0,T];H^s_{a,\theta}(\R)).
\end{equation*}
Moreover, for each $0\leq \mu<1$ fixed, the flow map data-to-solution $\phi\mapsto u_{\mu}$ is continuous in the $H^s_{a,\theta}$-norm. Additionally, there exists a function $\rho\in C([0,T];[0,\infty))$ such that
\begin{equation*}
\|u_{\mu}(t)\|_{H^s_{a,\theta}}^2\leq \rho(t), \hspace{0.5cm} t\in [0,T],
\end{equation*}
for all $0\leq \mu<1$.
\end{lemma}

Since the term $\partial_x D^{a+1}u$ in \eqref{gbo} involves a negative order operator, in contrast with \cite{AbBonaFellSaut1989,APlow,Iorio1986} some extra conditions and arguments are required to obtain Lemma \ref{existence}. For example, condition $\frac{2+a}{1+a}\leq \theta<-\frac{3}{2(1+a)}$ is needed to justify energy estimates in $L^2(\mathbb{R})$. Since the literature on the parabolic regularization method does not deal with equations with negative derivatives such as those given in \eqref{gbo}, we have decided to deduce Lemma \ref{existence} in the Appendix.

\begin{remark}\label{remarkona}
(i) In the proof of Lemma \ref{existence}, the condition $0<\theta<-\frac{3}{2(1+a)}$ is important to control the $\dot{H}^{(a+1)\theta}$-norm of the nonlinear term in \eqref{mugbo}. To see this, when $(a+1)\theta+1\geq 0$, we use the fractional Leibniz rule (see \eqref{leibfract}) to get
\begin{equation}\label{eqparabolicarg1}
\|D^{(a+1)\theta}\partial_x(u^2)\|=\|D^{(a+1)\theta+1}(u^2)\|\les \|u\|_{L^{\infty}}\|D^{(a+1)\theta+1} u\|\les \|u\|_{H^{\frac{3}{2}^{+}}}^2,    
\end{equation}
where we used the embedding $H^{\frac{3}{2}^{+}}(\mathbb{R})\hookrightarrow L^{\infty}(\mathbb{R})$ and the fact $\|u\|_{H_{a,\theta}^s}^2=\|u\|_{H^s}^2+\|D^{(a+1)\theta}u\|^2$. Now, when $(a+1)\theta+1<0$, using that $a>-1-\frac{3}{2\theta}$, and thus, $(a+1)\theta+\frac{3}{2}>0$, we deduce
\begin{equation}\label{eqparabolicarg2}
\begin{aligned}
  \|D^{(a+1)\theta}\partial_x(u^2)\|=\||\xi|^{(a+1)\theta+1}\widehat{u^2}(\xi)\|&\les \|\widehat{u^2}\|_{L^{\infty}}\||\xi|^{(a+1)\theta+1}\|_{L^2(|\xi|\leq 1)}+\|\widehat{u^2}\|_{L^2(|\xi|\geq 1)}\\
&\les \|u\|^2,  
\end{aligned}
\end{equation}
where we have also used that $\|\widehat{u^2}\|_{L^{\infty}}\les \|u^2\|_{L^1}=\|u\|^2$.

(ii) Using energy estimates, part of the proof of Lemma \ref{existence} (see Appendix) shows that $u_{\mu}$ converges to $u$ as $\mu \to 0^{+}$ in the sense of $C([0,T];L^2(\mathbb{R}))$.
\end{remark}


\begin{proof}[Proof of Theorem \ref{lwpw}]

We divide the proof into three main cases: $0<\theta \leq 1$, $1<\theta <\frac{-3}{2(a+1)}$, and  $\frac{-3}{2(a+1)}\leq \theta<\frac{1+2a}{2(1+a)}$. For the first two cases, we will deduce Theorem \ref{lwpw} by a contraction argument and a limiting process $\mu\to 0$ of solutions of \eqref{mugbo}. For the latter, we will use the results of case $0<\theta <\frac{-3}{2(a+1)}$, and some iteration. 

\underline{\bf Assume that $0<\theta \leq 1$}. We fix $s>\frac{3}{2}$ and consider $\phi\in Z_{s,\theta}^{a,\widetilde{\theta}}$, where if $\theta< \frac{2+a}{1+a}$, $\widetilde{\theta}=\frac{2+a}{1+a}$, and if $\theta\geq  \frac{2+a}{1+a}$, $\widetilde{\theta}=\theta$. For the definition of the spaces $H_{a,\widetilde{\theta}}^s(\mathbb{R})$, and $Z_{s,\theta}^{a,\widetilde{\theta}}$, see \eqref{Hspacedefi} and  \eqref{defiZspace}, respectively.  In virtue of Lemma \ref{existence}, we have a family $\{u_{\mu}\}_{0\leq \mu\leq 1}$ of solutions of \eqref{mugbo}, if $0<\mu<1$, and the solution of \eqref{gbo} if $\mu=0$, all of them in the same class  $C([0,T];H^s_{a,\widetilde{\theta}}(\R))$  for a common time $T>0$, and such that $u_{\mu}(0)=\phi$. We first show that for each $0<\mu< 1$, the solution $u_{\mu}$ is in the class $ C([0,T];Z_{s,\theta}^{a,\widetilde{\theta}})$. To get such a result, we will use the contraction mapping principle and uniqueness. We set
\begin{equation}
\mathcal{X}_{T_1,\theta}^s=\{v\in C([0,T_1];Z_{s,\theta}^{a,\widetilde{\theta}}) \; :\;\sup_{t\in [0,T_1]}\|v(t)-U_\mu(t)\phi\|_{Z_{s,\theta}^{a,\widetilde{\theta}}}\leq 2\|\phi\|_{Z_{s,\theta}^{a,\widetilde{\theta}}}\},
\end{equation}
endowed with the distance function
\begin{equation*}
 d(v,v_1):=\sup_{t\in [0,T_1]}\|v(t)-v_1(t)\|_{Z_{s,\theta}^{a,\widetilde{\theta}}}=\sup_{t\in [0,T_1]}\Big(\|v(t)-v_1(t)\|_{H_{a,\widetilde{\theta}}^s}+\||x|^{\theta}\big(v(t)-v_1(t)\big)\|_{L^2}\Big).  
\end{equation*}
We consider the following function $\Phi(v)$, with $v\in \mathcal{X}_{T_1,\theta}^s$, being given by the integral formulation of \eqref{mugbo}
\begin{equation}\label{muint}
\Phi (v)(t)=U_\mu(t)\phi -\frac12\int_{0}^{t}U_\mu(t-\tau)\partial_{x}v^2(\tau)d\tau,
\end{equation}
where the operators $\{U_{\mu}(t)\}_{t>0}$ are defined as in \eqref{musemi}.  In what follows, we will show that for $T_1>0$ small enough, $\Phi$ defines a contraction on the complete metric space $\mathcal{X}_{T_1,\theta}^s$.

We first remark that the $Z_{s,\theta}^{a,\widetilde{\theta}}$-norm of $U_{\mu}(t)\phi$ is well-defined. This is easy to check for the $H^{s}_{a,\widetilde{\theta}}$-component of both norms. In the weighted space $L^2(|x|^{2\theta}\, dx)$, we observe that the definition of the space $H_{a,\widetilde{\theta}}^s(\mathbb{R})$ yields  $D^{(1+a)\theta}\phi\in L^2(\mathbb{R})$ (which, by definition of $\widetilde{\theta}$, follows from the fact $H^s(\mathbb{R})\cap \dot{H}^{(1+a)\widetilde{\theta}}(\mathbb{R})\hookrightarrow H^s(\mathbb{R})\cap \dot{H}^{(1+a)\theta}(\mathbb{R})$). Hence, it follows that $\phi$ satisfies the hypothesis in Corollary \ref{Dmu}, from which we get that $U_{\mu}(t)\phi$ is defined in the space $L^2(|x|^{2\theta}\, dx)$, and it follows
\begin{equation*}
   \|U_{\mu}(t)\phi\|_{Z_{s,\theta}^{a,\widetilde{\theta}}}\lesssim \langle t\rangle\|\phi\|_{Z_{s,\theta}^{a,\widetilde{\theta}}}
\end{equation*}
with implicit constant independent of $0<\mu<1$. Moreover, one can also prove that $U_{\mu}(t)\phi$ defines a continuous curve in $Z_{s,\theta}^{a,\widetilde{\theta}}$.

Let us estimate $\Phi(v)$ in the space $\mathcal{X}_{T_1,\theta}^s$, where $v\in \mathcal{X}_{T_1,\theta}^s$. We first study the $H^s_{a,\widetilde{\theta}}$-norm. Using that $\{U_{\mu}(t)\}_{t\geq 0}$ are contracting operators, Proposition \ref{regularity}, and the fact that $H^{\frac{3}{2}^{+}}(\mathbb{R})$ is a Banach algebra, it is seen that
\begin{equation*}
\begin{aligned}
 \|\Phi(v)(t)-U_{\mu}(t)\phi\|_{H^s}\leq & \int_0^{T_1}\|U_{\mu}(t-\tau)(v^2)(\tau)\|_{H^{s+1}}\, d\tau\\
 \leq & c\Big(\int_0^{T_1}(1+(\mu(t-\tau))^{-\frac{1}{2}})\,d \tau\Big)\Big(\sup_{t\in [0,T_1]}\|v(t)\|_{H^{s}}\Big)^2\\
 \leq & cT_1^{\frac{1}{2}}\Big(T_1^{\frac{1}{2}}+\mu^{-\frac{1}{2}}\Big)\Big(\sup_{t\in [0,T_1]}\|v(t)\|_{H^{s}}\Big)^2.
\end{aligned}    
\end{equation*}
The argument in Remark \ref{remarkona} (i) establishes
\begin{equation}\label{des1Phi0}
\begin{aligned}
 \|D^{(1+a)\widetilde{\theta}}(\Phi(v)(t)-U_{\mu}(t)\phi)\|\leq & \int_0^{T_1}\|D^{(1+a)\widetilde{\theta}}\partial_x(v^2)(\tau)\|\, d\tau\\
 \leq & cT_1\Big(\sup_{t\in [0,T_1]}\|v(t)\|_{H^{s}}\Big)^2.
\end{aligned}    
\end{equation}
The previous two estimates complete the study of $\Phi$ in the space $H_{a,\widetilde{\theta}}^s(\mathbb{R})$. To control the $L^2(|x|^{2\theta})$-norm, we apply Corollary \ref{Dmu}, and the arguments in Remark \ref{remarkona} to get
\begin{equation}\label{des1Phi}
\begin{split}
\||x|^\theta(\Phi(v)(t)-U_\mu (t)\phi)\| 
&\leq c\langle T_1\rangle\int_0^{T_1} \Big( \|D^{(1+a)\theta}\p_x (v^2)(\tau)\|+\||x|^\theta \p_x (v^2)(\tau)\| \Big)\, d\tau\\
&\leq c\langle T_1\rangle\int_0^{T_1} \Big( \|v(\tau )\|_{H^s}^2+\|\partial_x v(\tau)\|_{L^{\infty}}\||x|^\theta v(\tau)\| \Big)\, d\tau \\
&\leq  cT_1\langle T_1\rangle \Big(\sup_{t\in [0,T_1]} \big(\|v(t)\|_{H^s}+\||x|^{\theta}v(\tau)\|\big)\Big)^2,
\end{split}
\end{equation}
where to bound $\|\partial_x v\|_{L^{\infty}}$, we used Sobolev embedding, and then Young's inequality.

Similarly, the previous arguments show
\begin{equation*}
\begin{aligned}
\|\Phi(v)(t)-\Phi(v_1)(t)\|_{H^s_{a,\widetilde{\theta}}}\leq cT_1^{\frac{1}{2}}\big(T^{\frac{1}{2}}_1+\mu^{-\frac{1}{2}}\big)&\big(\sup_{t\in[0,T_1]}\|v(t)\|_{H^s}+\sup_{t\in[0,T_1]}\|v_1(t)\|_{H^s}\big)\\
&\qquad \times\big(\sup_{t\in[0,T_1]}\|v(t)-v_1(t)\|_{H^s}\big),    
\end{aligned}
\end{equation*}
and it follows
\begin{equation*}
\begin{aligned}
\||x|^{\theta}\big(\Phi(v)(t)&-\Phi(v_1)(t)\big)\|\\
\leq cT_1\langle T_1\rangle &\Big(\sup_{t\in[0,T_1]}\big(\|v(t)\|_{H^s}+\||x|^{\theta}v(t)\|_{H^s}\big)+\sup_{t\in[0,T_1]}\big(\|v_1(t)\|_{H^s}+\||x|^{\theta}v_1(t)\|\big)\Big)\\
&\qquad \times\big(\sup_{t\in[0,T_1]}\|v(t)-v_1(t)\|_{H^s}+ \sup_{t\in[0,T_1]}\||x|^{\theta}\big(v(t)-v_1(t)\big)\|\big).   
\end{aligned}
\end{equation*}
Consequently, gathering our previous estimates,  and taking $T_1=T_1(\mu, \phi)>0$ sufficiently small such that
\begin{equation*}
\left\{\begin{aligned}
    &cT^{\frac{1}{2}}_1\big(T^{\frac{1}{2}}+\mu^{-\frac{1}{2}}\big)\langle T_1\rangle^2\|\phi\|_{Z_{s,\theta}^{a,\widetilde{\theta}}}<\frac{1}{2}, \\   
    &cT_1\langle T_1 \rangle^2 \|\phi\|_{Z_{s,\theta}^{a,\widetilde{\theta}}}<\frac{1}{2},
\end{aligned}\right.   
\end{equation*}
it follows that $\Phi:\mathcal{X}_{T_1,\theta}^s\to \mathcal{X}_{T_1,\theta}^s$ is a contraction, where $c>0$ above comes from our estimates and it is independent of $\mu$, $\phi$, and $T_1>0$. Therefore, Banach fixed point theorem guarantees that there exists a unique solution of the integral formulation of \eqref{mugbo} in the space $\mathcal{X}_{T_1,\theta}^s$, but such a solution is also a solution of \eqref{mugbo} in the class $C([0,T_1];H^s_{a,\widetilde{\theta}}(\mathbb{R}))$. It turns out from the uniqueness statement in Lemma \ref{existence} that $u_\mu\in C([0,T^{\ast}];L^2(|x|^{2\theta}\, dx))$, where $T^{\ast}=\min\{T,T_1\}$. Next, we show that the solution $u_{\mu}$ can be extended, if necessary, to the whole space $C([0,T];L^2(|x|^{2\theta}\, dx))$. To get such result, we have that if $T_1<T$ and $u_\mu\in C([0,T_1];L^2(|x|^{2\theta}\, dx))$, the arguments in Remark \ref{remarkona}, and \eqref{des1Phi} yield 
\begin{equation}\label{aftergronwall}
\begin{split}
\||x|^\theta u_\mu(t)\|\leq &  \; \||x|^{\theta}U_\mu (t)\phi\|+\int_0^t \||x|^{\theta}U_\mu (t-\tau)\p_x u_{\mu}^2\|d\tau\\
\leq & c\langle T\rangle\|\phi\|_{Z_{s,\theta}^{a,\widetilde{\theta}}}+cT\langle T\rangle\big(\sup_{t\in [0,T]}\rho(t)\big)+c\langle T \rangle\big(\sup_{t\in [0,T]}\rho(t)^{\frac{1}{2}}\big)\int_0^t\||x|^{\theta}u_{\mu}(\tau)\|\, d\tau,
\end{split}
\end{equation}
for all $t\in [0,T_1]$, and the function $\rho(t)$ as in Lemma \ref{existence}. Consequently, \eqref{aftergronwall}, and Gronwall's inequality show
\begin{equation}\label{uniformbound}
\sup_{t\in [0,T_1]}\|\lanx^{\theta}u_\mu(t)\|\leq C(\phi,T)<\infty
\end{equation}
for a constant $C(\phi,T)>0$ independent of $\mu$ and $T_1>0$. This in turn implies that $u_{\mu}$ can be extended to the class $C([0,T];L^2(|x|^{2\theta}\, dx))$, and \eqref{uniformbound} holds for $t\in [0,T]$. Additionally, since $u_{\mu}\to u$ in $C([0,T];L^2(\mathbb{R}))$ (see Remark \ref{remarkona} (ii)), we can take $\mu\to 0^{+}$ (using Fatou's lemma) in the inequality above to deduce 
\begin{equation*}
\sup_{t\in [0,T]}\|\lanx^{\theta}u(t)\|\leq C(\phi,T),
\end{equation*}
from which $u\in L^{\infty}([0,T];L^2(|x|^{2\theta}\, dx))$. The continuity $u\in C([0,T];L^2(|x|^{2\theta}\, dx))$ follows by approximation (e.g., see \cite{CunhaPastor2014}), and it can be obtained from energy estimates.  Continuous dependence also follows similarly. This completes the proof of Theorem \ref{lwpw} (i). 


\underline{\bf Assume that $1<\theta <\frac{-3}{2(1+a)}$}. Here, we fix $s\geq \max\{\frac{3}{2}^{+},\frac{5\theta(((1+a)(\theta-1)+1))}{3(2\theta-1)-\theta}\}$, and consider $\phi \in \mathcal{Z}_{s,\theta}^{a,\theta}$ (see \eqref{defiZspace2}). We will follow the same strategy as in the previous case. Thus,  by  Lemma \ref{existence}, let $\{u_{\mu}\}_{0\leq \mu\leq 1}$ be a family of solutions of \eqref{mugbo} if $0<\mu<1$, and the solution of \eqref{gbo} if $\mu=0$, all of them in the same class  $C([0,T];H^s_{a,\theta}(\R))$  for a common time $T>0$, and with the same initial data $u_{\mu}(0)=\phi$. We define
\begin{equation}
\overline{\mathcal{X}}_{T_1,\theta}^s=\{v\in C([0,T_1];\mathcal{Z}_{s,\theta}^{a,\theta}) \; :\;\sup_{t\in [0,T_1]}\|v(t)-U_\mu(t)\phi\|_{\mathcal{Z}_{s,\theta}^{a,\theta}}\leq 2\|\phi\|_{\mathcal{Z}_{s,\theta}^{a,\theta}}\},
\end{equation}
equipped with distance
\begin{equation*}
\begin{aligned}
\overline{d}(v,v_1):=&\sup_{t\in [0,T_1]}\|v(t)-v_1(t)\|_{\mathcal{Z}_{s,\theta}^{a,\theta}}\\
=&\sup_{[0,T_1]}\Big(\|v(t)-v_1(t)\|_{H_{a,\theta}^s}+\||x|^{\theta}\big(v(t)-v_1(t)\big)\|_{L^2}+\||x|^{\theta-1}D^{1+a}\big(v(t)-v_1(t)\big)\|_{L^2}\\
&+\|D^{(1+a)(\theta-1)}\big(x(v(t)-v_1(t))\big)\|_{L^2}\Big).     
\end{aligned}
\end{equation*}
We consider the same function $\Phi$ as in \eqref{muint}, but here it is given over $\overline{\mathcal{X}}_{T_1,\theta}^s$. 

We  first note that $U_{\mu}(t)\phi$ is defined in  $\mathcal{Z}_{s,\theta}^{a,\theta}$. The continuity of the semigroup $\{U_{\mu}(t)\}$ in $L^2$-spaces shows that the $H^s_{a,\theta}(\mathbb{R})$-norm of  $U_{\mu}(t)\phi$ is well-defined. Now, by the definition of $\mathcal{Z}_{s,\theta}^{a,\theta}$, we have $|x|^{\theta-1}D^{1+a}\phi$, $D^{(1+a)(\theta-1)}(x\phi)$, and  $D^{(1+a)\theta}\phi$ belong to $L^2(\R)$. Hence, it follows that $\phi$ satisfies the hypothesis in Corollary \ref{Dmu}, from which we get that $U_{\mu}(t)\phi$ is defined in the space $L^2(|x|^{2\theta}\, dx)$. Additionally, Corollary \ref{Dmu} also shows $D^{(1+a)(\theta-1)}\big(x U_{\mu}(t)\phi\big)\in L^2(\mathbb{R})$, $|x|^{\theta-1}D^{1+a} U_{\mu}(t) \phi \in L^2(\mathbb{R})$ provided the assumptions over $\phi$ hold. Summarizing, we have
\begin{equation*}
    \|U_\mu(t)\phi\|_{\mathcal{Z}_{s,\theta}^{a,\theta}}\lesssim \langle t\rangle^2 \|\phi\|_{\mathcal{Z}_{s,\theta}^{a,\theta}},
\end{equation*}
and the mapping $t\mapsto U_\mu(t)\phi\in  \mathcal{Z}_{s,\theta}^{a,\theta}$ is continuous.

Let us show the contraction property of $\Phi$ over $\overline{\mathcal{X}}_{T_1,\theta}^s$. By Remark \ref{remarkona}, the $H^s_{a,\theta}$-norm is estimated as in the case $0<\theta\leq 1$ above.  To study the $L^2(|x|^{2\theta}\, dx)$-norm, we use Corollary \ref{Dmu} to get
\begin{equation}\label{des1Phi2}
\begin{split}
\||x|^\theta(\Phi(v)(t)-U_\mu (t)\phi)\| 
\leq & c\langle T_1\rangle^2\int_0^{T_1} \Big( \|D^{(1+a)\theta}\p_x (v^2)(\tau)\|+\||x|^{\theta-1}D^{1+a}\p_x (v^2)(\tau)\|\\
&\hspace{1.5cm}+\|D^{(1+a)(\theta-1)}(x\p_x (v^2))(\tau)\|+\||x|^\theta \p_x (v^2)(\tau)\| \Big)\, d\tau.
\end{split}
\end{equation}
We continue with the estimate of the right-hand side of the expression above. For the moment, we omit the dependence on the time variable $\tau$. Notice that by Remark \ref{remarkona}, the first and last term above can be estimated as in \eqref{des1Phi}. Using Theorem \ref{stein}, the property \eqref{Leib} of $\mathcal{D}^{\theta-1}$, and writing $\xi=|\xi|\sgn(\xi)$, we have
\begin{equation*}
\begin{aligned}
\||x|^{\theta-1}&D^{1+a}\partial_x(v^2)\|\\
\les &\|D^{\theta-1}\big(|\xi|^{2+a} \sgn(\xi)\varphi\widehat{v^2}(\xi)\big)\| +\|D^{\theta-1}\big(|\xi|^{2+a} \sgn(\xi)(1-\varphi)\widehat{v^2}(\xi)\big)\|  \\
\les &\|\mathcal{D}^{\theta-1}\big(|\xi|^{2+a} \sgn(\xi)\varphi)\widehat{v^2}(\xi)\big)\|+\||\xi|^{2+a}\varphi \mathcal{D}^{\theta-1}(\widehat{v^2}\big)\|\\
&+\big(\||\xi|^{2+a}\sgn(\xi)(1-\varphi)\|_{L^{\infty}}+\|\mathcal{D}^{\theta-1}\big(|\xi|^{2+a}\sgn(\xi)(1-\varphi)\big)\|_{L^{\infty}}\big)\|J_{\xi}^{\theta-1}(\widehat{v^2})\|,
\end{aligned}    
\end{equation*}
where we used a function $\varphi\in C^{\infty}_0(\mathbb{R})$ with $\varphi(\xi)=1$, whenever $|\xi|\leq 1$. We proceed with the estimates of the right-hand side of the above inequality. Since $0<-(2+a)<\frac{1}{2}$, and $1<\theta<-\frac{3}{2(1+a)}\leq \frac{7}{2}+a$, we can apply Proposition \ref{localdecaylemma} to deduce
\begin{equation*}
\begin{aligned}
\|\mathcal{D}^{\theta-1}\big(|\xi|^{2+a} \sgn(\xi)\varphi)\widehat{v^2}(\xi)\big)\|\lesssim & \||\xi|^{3+a-\theta}\widehat{v^2}\|_{L^2(|\xi|\leq 1)}+\|\widehat{v^2}\|_{L^2(|\xi|\geq 1)}\\
\les & \||\xi|^{3+a-\theta}\|_{L^2(|\xi|\leq 1)}\|\widehat{v^2}\|_{L^{\infty}}+\|v^2\|\\
\les & \|v^2\|_{L^{1}}+\|v\|_{L^{\infty}}\|v\|\les \|v\|_{H^s}^2.
\end{aligned}    
\end{equation*}
Let $-(2+a)<\beta<\frac{1}{2}$, by H\"older's inequality, and Hardy–Littlewood–Sobolev estimate, we deduce
\begin{equation}\label{Newestimate0.1}
\begin{aligned}
\||\xi|^{2+a}\varphi \mathcal{D}^{\theta-1}(\widehat{v^2}\big)\|\les & \||\xi|^{2+a}\varphi \|_{L^{\frac{1}{\beta}}}\|\mathcal{D}^{\theta-1}(\widehat{v^2}\big)\|_{L^{\frac{2}{1-2\beta}}}  \\
\les & \|J^{\theta-1}(\widehat{v^2}\big)\|_{L^{\frac{2}{1-2\beta}}}\\
\les & \|J^{\theta-1+\beta}(\widehat{v^2}\big)\|\les \|v\|_{L^{\infty}}\|\langle x \rangle^{\theta}v\|,
\end{aligned}    
\end{equation}
where we also applied Theorem \ref{stein} to bound the fractional derivative $\mathcal{D}^{\theta-1}$ by $J^{\theta-1}$ in $L^{\frac{2}{1-2\beta}}(\mathbb{R})$. Notice that the last line above also provides an estimate for $\|J_{\xi}^{\theta-1}(\widehat{v^2})\|$. Summarizing, we arrive at
\begin{equation}\label{Newestimate1}
\begin{aligned}
\||x|^{\theta-1}D^{1+a}\partial_x(v^2)\|\les \|v\|_{H^s}^2+\|v\|_{H^s}\|\langle x \rangle^{\theta} v\|.
\end{aligned}    
\end{equation}
On the other hand, using Plancherel's identity and distributing the local derivative $\frac{\partial}{\partial \xi}$ provided by the weight $x$ in the frequency domain, it follows
\begin{equation*}
\begin{aligned}
\||D^{(1+a)(\theta-1)}(x\partial_x(v^2))\|\les & \||\xi|^{(1+a)(\theta-1)}\widehat{v^2}\|+\||\xi|^{(1+a)(\theta-1)+1}\frac{\partial }{\partial\xi}\widehat{v^2}\|.
\end{aligned}    
\end{equation*}
Using that $\theta<-\frac{3}{2(1+a)}<\frac{1+2a}{2(1+a)}$, and dividing into frequencies $|\xi|\leq 1$ and $|\xi|\geq 1$, we find
\begin{equation*}
 \||\xi|^{(1+a)(\theta-1)}\widehat{v^2}\|\les  \||\xi|^{(1+a)(\theta-1)}\|_{L^2(|\xi|\leq 1)}\|\widehat{v^2}\|_{L^{\infty}}+\|\widehat{v^2}\|\les \|v\|_{H^s}^2.
\end{equation*}
On the other hand, since $0<\theta<\frac{1+2a}{2(1+a)}$, we have $(1+a)(\theta-1)+1>0$. Thus, we distribute the local derivative to get
\begin{equation}\label{Newestimate1.1}
 \begin{aligned}
\||\xi|^{(1+a)(\theta-1)+1}\frac{\partial}{\partial \xi}\widehat{v^2}\|
\lesssim & \|\frac{\partial}{\partial \xi}\big(\langle \xi \rangle^{(1+a)(\theta-1)+1}\widehat{v^2}\big)\|+\|\frac{\partial}{\partial \xi}\big(\langle \xi \rangle^{(1+a)(\theta-1)+1}\big)\widehat{v^2}\|\\
\lesssim & \|J_{\xi}\big(\langle \xi \rangle^{{(1+a)(\theta-1)+1}}\widehat{v^2}\big)\|\\   
\les &\|\langle \xi\rangle^{s_1}\widehat{v^2}\|^{\frac{\theta_1-1}{\theta_1}}\|J_{\xi}^{\theta_1}\widehat{v^2}\|^{\frac{1}{\theta_1}}, 
\end{aligned}   
\end{equation}
where we applied complex interpolation with $\theta_1=\frac{(4s-1)}{2s}\theta$ and $s_1=\frac{\theta_1((1+a)(\theta-1)+1)}{(\theta_1-1)}$, and that $|\frac{\partial}{\partial \xi}\big(\langle \xi \rangle^{(1+a)(\theta-1)+1}\big)|\lesssim \langle \xi \rangle^{(1+a)(\theta-1)+1}$. Since, $\|J_{\xi}^{\theta_1}\widehat{v^2}\|=\|\langle x\rangle^{\theta_1}v^2\|$, Sobolev embedding, $H^{\frac{1}{4}}(\mathbb{R})\hookrightarrow L^{4}(\mathbb{R})$, and complex interpolation yield
\begin{equation}\label{Newestimate1.2}
\begin{aligned}
 \|\langle x\rangle^{\theta_1}v^2\|=\|\langle x\rangle^{\frac{\theta_1}{2}}v\|^2_{L^4}\les & \|J^{\frac{1}{4}}(\langle x\rangle^{\frac{\theta_1}{2}}v)\|^2 \\
 \les & \|\langle x \rangle^{\theta}v\|^{\frac{4s-1}{2s}}\|v\|_{H^s}^{\frac{1}{2s}}.
\end{aligned}    
\end{equation}
We have that $s_1\leq \frac{5\theta((1+a)(\theta-1)+1)}{3(2\theta-1)-\theta}\leq s$, then  since $H^s(\mathbb{R})$ is a Banach algebra, it is seen that
\begin{equation*}
 \|\langle \xi\rangle^{s_1}\widehat{v^2}\|\les \|J^s(v^2)\|_{L^2}\lesssim \|v\|_{H^s}^2. 
\end{equation*}
Collecting the previous estimates, we deduce
\begin{equation}\label{Newestimate2}
\begin{aligned}
\||D^{(1+a)(\theta-1)}(x\partial_x(v^2))\|\les & \|v\|_{H^s}^{\frac{2(\theta_1-1)}{\theta_1}} \big(\|\langle x\rangle^{\theta}v\|^{\frac{4s-1}{2s}}\|v\|_{H^{s}}^{\frac{1}{2s}}\big)^{\frac{1}{\theta_1}}\\
=& \|\langle x\rangle^{\theta}v\|^{\frac{1}{\theta}}\|v\|_{H^{s}}^{\frac{2((4s-1)\theta-2s)+1}{(4s-1)\theta}}.
\end{aligned}    
\end{equation}
We observe that one can use Young's inequality in the deduction of \eqref{Newestimate2} to obtain
\begin{equation*}
\|D^{(1+a)(\theta-1)}(x\partial_x(v^2))\|\les \|\langle x\rangle^{\theta}v\|^2+\|v\|_{H^s}^2.    
\end{equation*}
Plugging \eqref{Newestimate1}, and \eqref{Newestimate2} into \eqref{des1Phi2} establishes 
\begin{equation}\label{des1Phi3}
\begin{split}
\||x|^\theta(\Phi(v)(t)-U_\mu (t)\phi)\| 
&\leq  cT_1\langle T_1\rangle^2 \Big(\sup_{t\in [0,T_1]} \big(\|v(t)\|_{H^s}+\||x|^{\theta}v(\tau)\|\big)\Big)^2.
\end{split}
\end{equation}
Similarly,
\begin{equation*}
\begin{aligned}
\||x|^{\theta}\big(\Phi(v)(t)&-\Phi(v_1)(t)\big)\|\\
\leq cT_1\langle T_1\rangle^2&\Big(\sup_{t\in[0,T_1]}\big(\|v(t)\|_{H^s}+\||x|^{\theta}v(t)\|_{H^s}\big)+\sup_{t\in[0,T_1]}\big(\|v_1(t)\|_{H^s}+\||x|^{\theta}v_1(t)\|\big)\Big)\\
&\qquad \times\big(\sup_{t\in[0,T_1]}\|v(t)-v_1(t)\|_{H^s}+ \sup_{t\in[0,T_1]}\||x|^{\theta}\big(v(t)-v_1(t)\big)\|\big).   
\end{aligned}
\end{equation*}
It remains to estimate $|x|^{\theta-1}D^{1+a}\Phi(v)$, and $D^{(1+a)(\theta-1)}(x\Phi)(v)$. An application of the first inequality in Corollary \ref{Dmu}, together with \eqref{eqparabolicarg2} and \eqref{Newestimate1} yield
\begin{equation}\label{desPhi4}
\begin{aligned}
 \||x|^{\theta-1}&D^{1+a}(\Phi(v)(t)-U_\mu (t)\phi)\|\\
 \leq &\int_0^{t}\||x|^{\theta-1} U_{\mu}(t-\tau)D^{1+a}\partial_x(v^2)(\tau)\|\, d\tau\\
 \leq  & c\int_0^{t} \langle t-\tau\rangle\big(\|D^{(1+a)\theta}\partial_x(v^2)(\tau)\|+ \||x|^{\theta-1} D^{1+a}\partial_x(v^2)(\tau)\|\big)\, d\tau\\
 \leq & cT_1\langle T_1\rangle^2 \Big(\sup_{t\in [0,T_1]} \big(\|v(t)\|_{H^s}+\||x|^{\theta}v(\tau)\|\big)\Big)^2.
\end{aligned}    
\end{equation}
Using the identity
\begin{equation*}
  x U_{\mu}(t)f=U_{\mu}(t)\big(x-(2+a)t D^{1+a}-2\mu t \partial_x\big)f,
\end{equation*}
Corollary \ref{Dmu}, and \eqref{regulineq}, we deduce
\begin{equation}\label{desPhi5}
\begin{aligned}
 \|D^{(1+a)(\theta-1)} &\big(x(\Phi(v)(t)-U_\mu (t)\phi)\big)\|\\
 \leq &\int_0^{t}\|D^{(1+a)(\theta-1)}\Big( x U_{\mu}(t-\tau)\partial_x(v^2)(\tau)\Big)\|\, d\tau\\
 \leq & c\int_0^{t}\langle t-\tau\rangle \Big(\|D^{(1+a)(\theta-1)}(x\partial_x(v^2))(\tau)\|+\|D^{(1+a)\theta}\partial_x(v^2)(\tau)\|\\
 &+ \mu^{\frac{1}{2}}(t-\tau)^{-\frac{1}{2}} \|D^{(1+a)(\theta-1)}\partial_x(v^2)(\tau)\|\Big)\, d\tau\\
 \leq & cT_1^{\frac{1}{2}}\langle T_1\rangle^{\frac{3}{2}} \Big(\sup_{t\in [0,T_1]} \big(\|v(t)\|_{H^s}+\||x|^{\theta}v(\tau)\|\big)\Big)^2,
\end{aligned}    
\end{equation}
where we have also used \eqref{Newestimate2} and \eqref{eqparabolicarg2}. Summarizing
\begin{equation*}
\begin{aligned}
 \||x|^{\theta-1}D^{1+a}(\Phi(v)(t)-U_\mu (t)\phi)\| &+ \|D^{(1+a)(\theta-1)} \big(x(\Phi(v)(t)-U_\mu (t)\phi)\big)\| \\
 \leq & cT_1^{\frac{1}{2}}\langle T_1\rangle^{\frac{5}{2}} \Big(\sup_{t\in [0,T_1]} \big(\|v(t)\|_{H^s}+\||x|^{\theta}v(\tau)\|\big)\Big)^2,
\end{aligned}    
\end{equation*}
and in a similar fashion
\begin{equation*}
\begin{aligned}
 \||x|^{\theta-1} & D^{1+a}(\Phi(v)(t)-\Phi(v_1)(t))\| + \|D^{(1+a)(\theta-1)} \big(x(\Phi(v)(t)-\Phi(v_1)(t)\phi\big)\| \\
 \leq & cT_1^{\frac{1}{2}}\langle T_1\rangle^{\frac{5}{2}}\Big(\sup_{t\in[0,T_1]}\big(\|v(t)\|_{H^s}+\||x|^{\theta}v(t)\|_{H^s}\big)+\sup_{t\in[0,T_1]}\big(\|v_1(t)\|_{H^s}+\||x|^{\theta}v_1(t)\|\big)\Big)\\
&\qquad \times\big(\sup_{t\in[0,T_1]}\|v(t)-v_1(t)\|_{H^s}+ \sup_{t\in[0,T_1]}\||x|^{\theta}\big(v(t)-v_1(t)\big)\|\big).   
\end{aligned}    
\end{equation*}
Consequently, gathering our previous estimates,  and following a similar argument to that given in the previous cases $0<\theta\leq 1$, it follows that there exists $T_1=T_1(\mu, \phi)>0$ sufficiently small such that  $\Phi:\overline{\mathcal{X}}_{T_1,\theta}^s\to \overline{\mathcal{X}}_{T_1,\theta}^s$ is a contraction. Thus, for each $\mu>0$, there exists a unique solution of the integral formulation of \eqref{mugbo} in the space $\overline{\mathcal{X}}_{T_1,\theta}^s$, but such a solution is also a solution of \eqref{mugbo} in the class $C([0,T_1];H^s_{a,\theta}(\mathbb{R}))$. Uniqueness forces then $u_\mu\in C([0,T^{\ast}];\mathcal{Z}_{s,\theta}^{a,\theta})$, where $T^{\ast}=\min\{T,T_1\}$. Next, we extend, if necessary, to the whole time interval $[0,T]$, and take $\mu\to 0$.

The ideas in \eqref{des1Phi}, together with \eqref{Newestimate1} and \eqref{Newestimate2} establish
\begin{equation}\label{aftergronwall2}
\begin{split}
\||x|^\theta u_\mu(t)\|\leq &  \; \||x|^{\theta}U_\mu (t)\phi\|+\int_0^t \||x|^{\theta}U_\mu (t-\tau)\p_x u_{\mu}^2\|\, d\tau\\
\leq & c\langle T\rangle^2\Big(\|\phi\|_{Z_{s,\theta}^{a,\theta}}+\||x|^{\theta-1}D^{1+a}\phi\|+\|D^{(1+a)(\theta-1)}(x\phi)\|\big)\\
&+cT\langle T\rangle^2\big(\sup_{t\in [0,T]}\rho(t)+\big(\sup_{t\in [0,T]}\rho(t)\big)^{\frac{2((4s-1)\theta-2s)+1}{(4s-1)(\theta-1)}}\Big)\\
&+c\langle T \rangle^2\big(1+\sup_{t\in [0,T]}\rho(t)^{\frac{1}{2}}\big)\int_0^t\||x|^{\theta}u_{\mu}(\tau)\|\, d\tau.
\end{split}
\end{equation}
The arguments in \eqref{desPhi4}, \eqref{Newestimate1} and \eqref{Newestimate2} show
\begin{equation}\label{aftergronwall3}
\begin{split}
\||x|^{\theta-1}D^{1+a} u_\mu(t)\|\leq &  \; \||x|^{\theta-1} D^{1+a} U_\mu (t)\phi\|+\int_0^t \||x|^{\theta-1} D^{1+a} U_\mu (t-\tau)\p_x u_{\mu}^2\|\, d\tau\\
\leq & c\langle T\rangle \|\phi\|_{\mathcal{Z}_{s,\theta}^{a,\theta}}+cT\langle T\rangle \big(\sup_{t\in [0,T]}\rho(t)+\big(\sup_{t\in [0,T]}\rho(t)\big)^{\frac{2((4s-1)\theta-2s)+1}{(4s-1)(\theta-1)}}\Big)\\
&+c\langle T \rangle\big(1+\sup_{t\in [0,T]}\rho(t)^{\frac{1}{2}}\big)\int_0^t\||x|^{\theta}u_{\mu}(\tau)\|\, d\tau.
\end{split}
\end{equation}
Finally, using \eqref{desPhi5},  \eqref{Newestimate1} and \eqref{Newestimate2}, we find
\begin{equation}\label{aftergronwall4}
\begin{split}
\|D^{(1+a)(\theta-1)} &\big(x u_\mu\big)(t)\|\\
\leq &  \||D^{(1+a)(\theta-1)}\big(x U_\mu (t)\phi\big)\|+\int_0^t \||D^{(1+a)(\theta-1)}\big(x U_\mu (t-\tau)\p_x u_{\mu}^2\big)\|\, d\tau\\
\leq & c\langle T\rangle \|\phi\|_{\mathcal{Z}_{s,\theta}^{a,\theta}}+cT^{\frac{1}{2}}\langle T\rangle^{\frac{3}{2}}\big(\sup_{t\in [0,T]}\rho(t)+\big(\sup_{t\in [0,T]}\rho(t)\big)^{\frac{2((4s-1)\theta-2s)+1}{(4s-1)(\theta-1)}}\Big)\\
&+c\langle T \rangle \int_0^t\||x|^{\theta}u_{\mu}(\tau)\|\, d\tau.
\end{split}
\end{equation}
Consequently, defining
\begin{equation}
   \mathcal{G}(t):=\||x|^\theta u_\mu(t)\|+\||x|^{\theta-1}D^{1+a} u_\mu(t)\|+\||x|^{\theta-1}D^{1+a} u_\mu(t)\|,
\end{equation}
the previous inequalities establish
\begin{equation*}
\begin{aligned}
 \mathcal{G}(t)\leq  & c\langle T\rangle^2 \|\phi\|_{\mathcal{Z}_{s,\theta}^{a,\theta}}+cT^{\frac{1}{2}}\langle T\rangle^{\frac{5}{2}}\big(\sup_{t\in [0,T]}\rho(t)+\big(\sup_{t\in [0,T]}\rho(t)\big)^{\frac{2((4s-1)\theta-2s)+1}{(4s-1)(\theta-1)}}\Big)\\
&+c\langle T \rangle^2\big(1+\sup_{t\in [0,T]}\rho(t)^{\frac{1}{2}}\big)\int_0^t \mathcal{G}(\tau)\, d\tau,    
\end{aligned}
\end{equation*}
for all $t\in [0,T_1]$.  Gronwall's inequality now yields
\begin{equation}\label{uniformbound2}
\||x|^\theta u_\mu(t)\|+\||x|^{\theta-1}D^{1+a} u_\mu(t)\|+\||x|^{\theta-1}D^{1+a} u_\mu(t)\|\leq C(\phi,T)<\infty,
\end{equation}
for each $t\in [0,T_1]$, a constant $C(\phi,T)>0$ independent of $\mu$ and $T_1>0$. Since we already know that $u_{\mu}\in C([0,T];H^s_{a,\theta}(\mathbb{R}))$, \eqref{uniformbound2} shows that $u_{\mu}$ can be extended to the class $C([0,T];\mathcal{Z}_{s,\theta}^{a,\theta})$, and \eqref{uniformbound2} holds for $t\in [0,T]$. Convergence $u_{\mu}\to u$ in $C([0,T];L^2(\mathbb{R}))$ implies that taking $\mu\to 0^{+}$ in \eqref{uniformbound2} one infers
\begin{equation*}
u\in L^{\infty}([0,T];\mathcal{Z}_{s,\theta}^{a,\theta}).
\end{equation*}
Since the mapping $t\mapsto U(t)\phi\in \mathcal{Z}_{s,\theta}^{a,\theta}$ is continuous, the previous fact and using similar estimates to \eqref{aftergronwall2}, \eqref{aftergronwall3} and \eqref{aftergronwall4}, we deduce $u\in C([0,T];\mathcal{Z}_{s,\theta}^{a,\theta})$. Finally, the continuous dependence can be obtained using a differential inequality similar to that obtained for $\mathcal{G}(t)$ (i.e., arguing as in \eqref{aftergronwall2}-\eqref{aftergronwall4}), together with the continuous dependence on $H^s_{a,\theta}(\mathbb{R})$ given by Lemma \ref{existence}. To avoid repetitions, we omit further details.


\underline{\bf Assume that $\frac{-3}{2(1+a)}\leq \theta<\frac{1+2a}{2(1+a)}$}. 
An inspection of the arguments in Remark \ref{remarkona}, tells us that $\theta<\frac{-3}{2(1+a)}$ appears to be a candidate to be the decay limit allowed by solutions of \eqref{gbo}. However, we will show that the integral formulation of \eqref{gbo} allows larger decay as long as it is less than $\frac{1+2a}{2(1+a)}$. Since the previous approximation argument with \eqref{mugbo} focuses on condition $\theta<\frac{-3}{2(1+a)}$, in this part, we must use a different strategy to achieve the desired results. Firstly, we show that the integral formula allows faster decay than the linear part. Since such a result is valid for arbitrary $0<\theta<\frac{1+2a}{2(1+a)}$, we will establish a more general result.

\begin{proposition}\label{extradecaycond}
Let $-\frac{5}{2}<a<-2$, $0<\theta<\frac{1+2a}{2(1+a)}$. Consider the following restriction on $s$: $s>2$ if $0<\theta<1$, $s\geq\max\{\frac{3}{(5+2a)},2^{+}\}$, if $1\leq \theta<-\frac{3}{2(1+a)}$,  and $s\geq \max\{-\frac{9+6a}{2(5+2a)},2^{+}\}$, if  $-\frac{3}{2(1+a)}\leq \theta<\frac{1+2a}{2(1+a)}$. Let $u\in C([0,T];H^s(\mathbb{R}))\cap L^{\infty}([0,T];L^2(|x|^{2\theta}\, dx))$.  Then there exists $0<\theta<\theta_1<\frac{1+2a}{2(1+a)}$ such that
\begin{equation*}
\int_{0}^t U(t-\tau)(\partial_x u^2)(\tau)\, d\tau\in L^{\infty}([0,T]; L^2(|x|^{2\theta_1}\, dx)).    
\end{equation*}    
In the case $\theta=\Big(-\frac{3}{2(1+a)}\Big)^{-}=-\frac{3}{2(1+a)}-\epsilon$, for $\epsilon>0$ is sufficiently small and $s\geq \max\{-\frac{9+6a}{2(5+2a)},2^{+}\}$, one can take $\theta_1=-\frac{3}{2(1+a)}$. Moreover, in the case, $-\frac{3}{2(1+a)}\leq \theta<\frac{1+2a}{2(1+a)}$, one can take any $\theta_1$ such that $\theta<\theta_1<\min\{\frac{2(s-1)}{s}\theta,\frac{1+2a}{2(1+a)}\}$.
\end{proposition}

\begin{proof}
 When $0<\theta <  1$, following the ideas in \eqref{des1Phi}, which depend on Lemma \ref{Dthetau} and  Remark \ref{remarkona}, it is enough to show that there exists $0<\theta<\theta_1<1$ for which
 \begin{equation}\label{eqclaim1}
 \partial_x (u^2)\in L^{\infty}([0,T]; L^2(|x|^{2\theta_1}\, dx)).    
 \end{equation}
Now, setting $\theta<\theta_1<\min\{1,\frac{2(s-1)}{s}\theta\}$, since $H^1(\mathbb{R})$ is a Banach algebra,  complex interpolation shows 
\begin{equation}\label{eqclaim1.1}
\begin{aligned}
\|\langle x\rangle^{\theta_1}\partial_x(u^2)\|\les \|\big(\langle x\rangle^{\frac{(s-1)\theta}{s}}u\big)^2\|_{H^1}\les & \|\langle x\rangle^{\frac{(s-1)\theta}{s}}u\|_{H^1}^2   \\
\les &  \|\langle x\rangle^{\theta}u\|^2+\|J^s u\|^2.
\end{aligned}    
\end{equation}
This  establishes \eqref{eqclaim1}, and in turn the case $0<\theta < 1$. 

Next, we deal with the case $1\leq \theta<-\frac{3}{2(1+a)}$. By the argument in  estimate \eqref{des1Phi2}, Remark \ref{remarkona}, and the previous case, we are reduced to show that there exists $\theta<\theta_1<\min\{\frac{1+2a}{2(1+a)}, \frac{2(s-1)}{s}\theta\}$ such that
\begin{equation}\label{eqclaim1.0}
|x|^{\theta_1-1}D^{1+a}\partial_x (u^2),\, \, D^{(1+a)(\theta_1-1)}(x\partial_x (u^2))\in L^{\infty}([0,T]; L^2(\mathbb{R})).    
 \end{equation}
For any  $\theta<\theta_1<\min\{-\frac{3}{2(1+a)}, \frac{2(s-1)}{s}\theta\}$, by similar consideration in the deduction of \eqref{Newestimate1} (for more details see, \eqref{Newestimate0.1}), and \eqref{eqclaim1.1} we find
\begin{equation*}
\begin{aligned}
   \||x|^{\theta_1-1}D^{1+a}\partial_x(u^2)\|\les &\|u\|_{H^s}^2+\|\langle x \rangle^{\theta_1-\frac{1}{2}}u^2\|\\
  \les &\|u\|_{H^s}^2+\|\big(\langle x \rangle^{\frac{(s-1)\theta}{s}}u\big)^2\|_{H^1} \\
  \les &\|\langle x\rangle^{\theta}u\|^2+\|J^s u\|^2.
\end{aligned}
\end{equation*}
As in \eqref{Newestimate2} and  \eqref{Newestimate1.1},
\begin{equation*}
 \begin{aligned}
  \|D^{(1+a)(\theta_1-1)}(x\partial_x(u^2))\|\les &\|u\|_{H^s}^2+\|J_{\xi}\big(\langle \xi \rangle^{{(1+a)(\theta_1-1)+1}}\widehat{u^2}\big)\|\\
\les &\|u\|_{H^s}^2+\|\langle \xi\rangle^{\frac{3((1+a)(\theta_1-1)+1)}{(5+2a)}}\widehat{u^2}\|+\|J_{\xi}^{-\frac{3}{2(1+a)}}\widehat{u^2}\|,
 \end{aligned}   
\end{equation*}
where we have also used interpolation Lemma \ref{interx}. Hence, for any $1\leq \theta<\theta_1<\min\{-\frac{3}{2(1+a)}, \frac{2(s-1)}{s}\theta\}$, we have  $\frac{3((1+a)(\theta_1-1)+1)}{(5+2a)}\leq \frac{3}{(5+2a)}\leq s$. Now, changing to the frequency domain and using interpolation, a similar argument to \eqref{Newestimate1.2} establishes
\begin{equation*}
\begin{aligned}
\|J_{\xi}^{-\frac{3}{2(1+a)}}\widehat{u^2}\|\sim \|\langle x\rangle^{-\frac{3}{2(1+a)}}u^2\|\les \|\langle x\rangle^{\theta}u\|^2+\|J^{\frac{(1+a)\theta}{4(1+a)\theta+3}}u\|^2.    
\end{aligned}    
\end{equation*}
Since $\frac{(1+a)\theta}{4(1+a)\theta+3}\leq 2<s$, the previous estimates allow us to conclude
\begin{equation*}
 \begin{aligned}
  \|D^{(1+a)(\theta_1-1)}(x\partial_x(u^2))\|\les\|\langle x\rangle^{\theta}u\|_{L^2}^2+\|J^s u\|_{L^2}^2.
 \end{aligned}   
\end{equation*}
This completes the proof of the present case.

We continue with the remaining case $-\frac{3}{2(1+a)}\leq \theta<\frac{1+2a}{2(1+a)}$. Here, we assume that $\max\{-\frac{9+6a}{2(5+2a)},2^{+}\}\leq s$. Since $\theta>1$, in view of the identity $xU(t)f=U(t)(xf-(2+a)tD^{1+a}f)$, we are reduced to show that there exists $\theta<\theta_1<\frac{1+2a}{2(1+a)}$ such that 
\begin{equation}\label{thingstoprove0}
   \int_0^{t} U(t-\tau)(x\partial_x u^2-(2+a)(t-\tau)\partial_x D^{1+a}u^2)(\tau)\, d\tau\in  L^{\infty}([0,T];L^2(|x|^{2(\theta_1-1)}\, dx)). 
\end{equation}
Taking $\theta<\theta_1<\min\{\frac{2(s-1)}{s}\theta,\frac{1+2a}{2(1+a)}\}$ fixed, we will show that \eqref{thingstoprove0} holds for this weight side $\theta_1$. Firstly, we will deduce that 
\begin{equation*}
   \int_0^{t} U(t-\tau)(x\partial_x u^2)(\tau)\, d\tau\in  L^{\infty}([0,T];L^2(|x|^{2(\theta_1-1)}\, dx)). 
\end{equation*}
Since $-\frac{3}{2(1+a)}<\theta_1<\frac{1+2a}{2(1+a)}$ implies that $0<\theta_1-1<1$, using Lemma \ref{Dthetau}, it follows that the previous statement holds if we show 
\begin{equation}\label{thingstoprove}
 \langle x\rangle^{\theta_1} \partial_x (u^2),\,\, D^{(1+a)(\theta_1-1)}(x\partial_x u^2)\in L^{\infty}([0,T];L^2(\mathbb{R})).    
\end{equation}
When $\theta<\theta_1<2\theta$, arguing as in \eqref{eqclaim1.1}, we deduce
\begin{equation*}
 \|\langle x\rangle^{\theta_1}\partial_x (u^2)\|\les \|\langle x\rangle^{\theta}u\|^2+\|J^{\frac{2\theta}{2\theta-\theta_1}}u\|^2.  
\end{equation*}
Thus, taking $\theta<\theta_1<\frac{2(s-1)}{s}\theta$ with $s>2$, we have $\|J^{\frac{2\theta}{2\theta-\theta_1}}u\|\leq \|u\|_{H^s}$. Following the estimates leading to \eqref{Newestimate2}, we have
\begin{equation*}
\begin{aligned}
   \|D^{(1+a)(\theta_1-1)}(x\partial_x(u^2))\|\les &  \|u\|_{H^s}^2+\|J_{\xi}\big(\langle \xi \rangle^{(1+a)(\theta_1-1)+1}\widehat{u^2}\big)\|\\
   \les &  \|u\|_{H^s}^2+\|J^{\frac{\theta_1((1+a)(\theta_1-1)+1)}{\theta_1-1}}u^2\|+\|\langle x \rangle^{\theta_1} u^2\|.
\end{aligned}    
\end{equation*}
Since the function $f(\theta_1)=\frac{\theta_1((1+a)(\theta_1-1)+1)}{\theta_1-1}$ is non-increasing, $f(\theta_1)\leq f(\frac{-3}{2(1+a)})=-\frac{9+6a}{2(5+2a)}\leq s$. Thus, we conclude 
\begin{equation*}
      \|D^{(1+a)(\theta_1-1)}(x\partial_x(u^2))\| \les  \|u\|_{H^s}^2+\|\langle x \rangle^{\theta} u\|^2. 
\end{equation*}
Given that $u\in C([0,T];H^s(\mathbb{R}))\cap L^{\infty}([0,T];L^2(|x|^{2\theta}\, dx))$, the estimates deduced above imply \eqref{thingstoprove}. 
Going back to \eqref{thingstoprove0}, it only remains to show
\begin{equation}\label{thingstoprove2}
   \int_0^{t} U(t-\tau)((t-\tau)\partial_x D^{1+a}u^2)(\tau)\, d\tau\in  L^{\infty}([0,T];L^2(|x|^{2(\theta_1-1)}\, dx)). 
\end{equation}
Taking the Fourier transform and using a function $\varphi\in C^{\infty}_0(\mathbb{R})$ with $\varphi(\xi)=1$, when $|\xi|\leq 1$, we write
\begin{equation}\label{decompositionApart}
\begin{aligned}
\int_0^{t} &e^{i(t-\tau)\xi|\xi|^{1+a}}(t-\tau)\xi|\xi|^{1+a}\widehat{u^2}(\tau)\, d\tau\\
=&\Big(\int_0^{t}e^{i(t-\tau)\xi|\xi|^{1+a}}(t-\tau) \xi|\xi|^{1+a}\varphi(\xi)\widehat{u^2}(0,\tau)\, d\tau \Big)\\
&+\int_0^{t}e^{i(t-\tau)\xi|\xi|^{1+a}}(t-\tau)\xi|\xi|^{1+a}\varphi(\xi)\big(\widehat{u^2}(\xi,\tau)-\widehat{u^2}(0,\tau)\big)\, d\tau\\
&+\int_0^{t}e^{i(t-\tau)\xi|\xi|^{1+a}}(t-\tau)\xi|\xi|^{1+a}(1-\varphi(\xi))\widehat{u^2}(\xi,\tau)\, d\tau\\
=:&\mathcal{A}_1+\mathcal{A}_2+\mathcal{A}_3.
\end{aligned}   
\end{equation}
We have that \eqref{thingstoprove2} is equivalent to shows that $\mathcal{A}_1, \mathcal{A}_2, \mathcal{A}_3 \in L^{\infty}([0,T];H^{\theta_1-1}_{\xi}(\mathbb{R}))$. Using \eqref{Leibh} and the fact that $a<-2$ implies that the function $e^{i(t-\tau)\xi|\xi|^{1+a}}\xi|\xi|^{1+a}$ is bounded with bounded derivative outside  the origin, we get 
\begin{equation*}
\begin{aligned}
\|D^{\theta_1-1}\mathcal{A}_3\|\les & \int_0^{T}\Big(\|\mathcal{D}^{\theta_1-1}\big(e^{i(t-\tau)\xi|\xi|^{1+a}}(t-\tau)\xi|\xi|^{1+a}(1-\varphi(\xi))\big)\widehat{u^2}(\xi,\tau)\|\\
&\qquad+\|(t-\tau)\xi|\xi|^{1+a}(1-\varphi(\xi))\mathcal{D}^{\theta_1-1}(\widehat{u^2}(\xi,\tau))\|\Big)\, d\tau\\
\les & \int_0^{T}\Big(\|\partial_x\big(e^{i(t-\tau)\xi|\xi|^{1+a}}(t-\tau)\xi|\xi|^{1+a}(1-\varphi(\xi))\big)\|_{L^{\infty}}\\
&\qquad +\|(t-\tau)\xi|\xi|^{1+a}(1-\varphi(\xi))\|_{L^{\infty}}\big)\|J_{\xi}^{\theta_1-1}(\widehat{u^2}(\xi,\tau))\|\, d\tau\\
 \les & \langle T\rangle^2\Big(\sup_{t\in [0,T]}\|u(t)\|_{H^s}^2+\sup_{t\in[0,T]}\|\langle x\rangle^{\theta}u(t)\|^2\Big),
\end{aligned}    
\end{equation*}
where we have also used that $ \|J_{\xi}^{\theta_1-1}(\widehat{u^2}(\tau))\|\sim \|\langle x\rangle^{\theta_1-1}u^2(\tau)\|\les \|u(\tau)\|_{L^{\infty}}\|\langle x\rangle^{\theta}u(\tau)\|$, and Sobolev embedding. Before we deal with $\mathcal{A}_2$, we write
\begin{equation*}
 \begin{aligned}
 \widehat{u^2}(\xi,\tau)-\widehat{u^2}(0,\tau)=\xi\int_0^{1}\frac{\partial}{\partial \xi}\widehat{u^2}(\sigma\xi,\tau)\, d\sigma.    
 \end{aligned}   
\end{equation*}
Then, Sobolev embedding $H^1(\mathbb{R})\hookrightarrow L^{\infty}(\mathbb{R})$ and a change of variables yield
\begin{equation*}
\begin{aligned}
\Big\| \int_0^{1} \frac{\partial}{\partial \xi}\widehat{u^2}(\sigma\xi,\tau)\, d\sigma\Big\|_{L^{\infty}}\les & \int_0^1 \Big(\Big\|\frac{\partial}{\partial \xi}\widehat{u^2}(\sigma\xi,\tau)\Big\|+ \sigma\Big\|\frac{\partial^2}{\partial \xi^2}\widehat{u^2}(\sigma\xi,\tau)\Big\|\Big)\, d\sigma \\
\les & \int_0^1 \Big(\sigma^{-\frac{1}{2}}\Big\|\frac{\partial}{\partial \xi}\widehat{u^2}(\xi,\tau)\Big\|+ \sigma^{\frac{1}{2}}\Big\|\frac{\partial^2}{\partial\xi^2}\widehat{u^2}(\xi,\tau)\Big\|\Big)\, d\sigma.
\end{aligned}    
\end{equation*}
Using Plancherel's identity and similar ideas as in \eqref{Newestimate1.2}, one deduces
\begin{equation}\label{eqintegralestimate}
\begin{aligned}
\Big\|\frac{\partial}{\partial\xi}\widehat{u^2}(\xi,\tau)\Big\|+\Big\|\frac{\partial^2}{\partial\xi^2}\widehat{u^2}(\xi,\tau)\Big\|\les \|\langle x\rangle^{2}u^2\|= &\|\langle x\rangle u\|^{2}_{L^4}\\
\les & \|\langle x\rangle^{\theta} u\|^{\frac{2}{\theta}}\|J^{\frac{\theta}{4(\theta-1)}}u\|^{\frac{2(\theta-1)}{\theta}}.
\end{aligned}    
\end{equation}
Notice that $\theta\geq \frac{-3}{2(1+a)}$ and $s>2$ imply that $\|J^{\frac{\theta}{4(\theta-1)}}u\|\les \|J^{s}u\|$. We conclude that
\begin{equation}\label{integralestimate}
\begin{aligned}
\Big\| \int_0^{1} \frac{\partial}{\partial\xi}\widehat{u^2}(\sigma\xi,\tau)\, d\sigma\Big\|_{L^{\infty}}
\les  \sup_{t\in [0,T]}\|\langle x\rangle^{\theta} u(t)\|^{2}+\sup_{t\in [0,T]}\|u(t)\|_{H^s}^{2}.
\end{aligned}    
\end{equation}
We proceed with the study of $\mathcal{A}_2$. Using Theorem \ref{stein} and the properties of the fractional derivative, we get
\begin{equation*}
 \begin{aligned}
 \|D^{(\theta_1-1)}\big(\mathcal{A}_2\big)\|\les & \int_0^{t}  \Big\|\mathcal{D}^{(\theta_1-1)}\big(e^{i(t-\tau)\xi|\xi|^{1+a}}\big)(t-\tau)|\xi|^{3+a}\varphi(\xi)\Big(\int_0^1\frac{\partial}{\partial\xi}\widehat{u^2}(\sigma\xi,\tau)\, d\sigma\Big) \Big\|\, d\tau \\
 &+\int_0^{t}  \Big\|(t-\tau)D^{(\theta_1-1)}\big(\sgn(\xi)|\xi|^{3+a}\varphi(\xi)\Big(\int_0^1\frac{\partial}{\partial\xi}\widehat{u^2}(\sigma\xi,\tau)\, d\sigma\Big) \big)\Big\|\, d\tau\\
 =:&\mathcal{A}_{2,1}+\mathcal{A}_{2,2}.
 \end{aligned}   
\end{equation*}
Proposition \ref{pontualn}, the fact that $|\xi|^{3+a+(1+a)(\theta_1-1)}\varphi\in L^2(\mathbb{R})$,  and \eqref{integralestimate} yield
\begin{equation*}
\begin{aligned}
\mathcal{A}_{2,1}\les &   \int_0^{t}\|(t-\tau)|\xi|^{3+a+(1+a)(\theta_1-1)}\varphi\|\Big\| \int_0^{1} \frac{\partial}{\partial\xi}\widehat{u^2}(\sigma\xi,\tau)\, d\sigma\Big\|_{L^{\infty}}\, d\tau\\
\les & \langle T\rangle^2\Big( \sup_{t\in [0,T]}\|\langle x\rangle^{\theta} u(t)\|^{2}+\sup_{t\in [0,T]}\|u(t)\|_{H^s}^{2}\Big).    
\end{aligned}
\end{equation*}
To estimate $\mathcal{A}_{2,2}$, we commute with the derivative $D^{(\theta_1-1)}$ to get 
\begin{equation*}
\begin{aligned}
\Big\|D^{(\theta_1-1)}&\left(\sgn(\xi)|\xi|^{3+a}\varphi(\xi)\Big(\int_0^1\frac{\partial}{\partial \xi}\widehat{u^2}(\sigma\xi,\tau)\, d\sigma\Big) \right)\Big\|  \\
\les & \Big\|\left[D^{(\theta_1-1)},\big(\sgn(\xi)|\xi|^{3+a}\varphi(\xi)\big)\right]\Big(\int_0^1\frac{\partial}{\partial\xi}\widehat{u^2}(\sigma\xi,\tau)\, d\sigma\Big) \Big\|\\
&+\Big\||\xi|^{3+a}\varphi(\xi)D^{(\theta_1-1)}\Big(\int_0^1\frac{\partial}{\partial\xi}\widehat{u^2}(\sigma\xi,\tau)\, d\sigma\Big)\Big\|.
\end{aligned}    
\end{equation*}
Let us estimate each factor on the right-hand side of the inequality above. By the commutator estimate in Proposition \ref{comuDs1} and  \eqref{integralestimate}, we infer
\begin{equation*}
\begin{aligned}
\Big\|\left[D^{(\theta_1-1)},\big(\sgn(\xi)|\xi|^{3+a}\varphi(\xi)\big)\right] &\Big(\int_0^1\frac{\partial}{\partial \xi}\widehat{u^2}(\sigma\xi,\tau)\, d\sigma\Big) \Big\|\\
\les & \|D^{(\theta_1-1)}\big(\sgn(\xi)|\xi|^{3+a}\varphi)\|\Big\|\int_0^1\frac{\partial}{\partial \xi}\widehat{u^2}(\sigma\xi,\tau)\, d\sigma\Big\|_{L^{\infty}} \\
\les & \sup_{t\in [0,T]}\|\langle x\rangle^{\theta} u(t)\|^{2}+\sup_{t\in [0,T]}\|u(t)\|_{H^s}^{2},
\end{aligned}    
\end{equation*}
where the fact that $\theta_1-1<(3+a)+\frac{1}{2}$ (recall that $\theta_1<\frac{1+2a}{2(1+a)}$) implies $\mathcal{D}^{(\theta-1)}\big(\sgn(\xi)|\xi|^{3+a}\varphi(\xi)\big)\in L^2(\mathbb{R})$, which is a consequence of \cite[Proposition 2.9]{FLP1}. Now, $3+a>0$ shows us that
\begin{equation*}
\begin{aligned}
\Big\||\xi|^{3+a}\varphi(\xi)D^{(\theta_1-1)}\Big(\int_0^1\frac{\partial}{\partial \xi}\widehat{u^2}(\sigma\xi,\tau)\, d\sigma\Big) \Big\|\les & \||\xi|^{3+a}\varphi\|_{L^{\infty}}\int_0^1 \sigma^{(\theta_1-\frac{3}{2})} \|D^{(\theta_1-1)}\frac{\partial}{\partial \xi}\widehat{u^2}(\xi,\tau)\|\,  d\sigma  \\
\les &\|\langle x\rangle^{\theta_1}u^2(\tau)\|.
\end{aligned}    
\end{equation*}
The estimate above follows as a consequence of \eqref{eqclaim1.1}. Accordingly, collecting the previous results, we deduce
\begin{equation*}
\begin{aligned}
\|D^{\theta_1-1}(\mathcal{A}_2)\|\les  \langle T\rangle^2\Big( \sup_{t\in [0,T]}\|\langle x\rangle^{\theta} u(t)\|^{2}+\sup_{t\in [0,T]}\|u(t)\|_{H^s}^{2}\Big).    
\end{aligned}    
\end{equation*}
Finally, we proceed with the analysis of $\mathcal{A}_1$. The $L^2$-conservation law implies $\widehat{u^2}(0,\tau)=\|u(\tau)\|^2=\|\phi\|^2$, for all $\tau\in[0,T]$. It then follows 
\begin{equation}\label{decompositionApart2}
 \begin{aligned}
 \mathcal{A}_1&= -\frac{\big(t(i\xi|\xi|^{1+a})-1\big)e^{it\xi|\xi|^{1+a}}+1}{\xi|\xi|^{1+a}} \varphi(\xi)\|\phi\|^2\\
 &=-ite^{it\xi|\xi|^{1+a}}\varphi(\xi)\|\phi\|^2-(e^{it\xi|\xi|^{1+a}}-1)(\xi|\xi|^{1+a})^{-1}\varphi(\xi)\|\phi\|^2\\
  &=:-ite^{it\xi|\xi|^{1+a}}\varphi(\xi)\|\phi\|^2-\mathcal{A}_{1,2}.
 \end{aligned}   
\end{equation}
We first show that $\mathcal{A}_{1,2}$ is in the space $H^{\theta_1-1}(\mathbb{R})$ uniformly with respect to $t\in[0,T]$. Let $\widetilde{\varphi} \in C^{\infty}_0(\mathbb{R})$ be such that $\widetilde{\varphi}\varphi=\varphi$ and write
\begin{equation*}
\begin{aligned}
\|\mathcal{D}^{(\theta_1-1)}&\big((e^{it_2\xi|\xi|^{1+a}}-1)(\xi|\xi|^{1+a})^{-1}\widetilde{\varphi(\xi)}\varphi(\xi)\big)\|\\
\les & \|\mathcal{D}^{(\theta_1-1)}\big((e^{it_2\xi|\xi|^{1+a}}-1)\big)(\xi|\xi|^{1+a})^{-1}\varphi(\xi)\|   +\|\mathcal{D}^{(\theta_1-1)}\big((\xi|\xi|^{1+a})^{-1}\widetilde{\varphi}(\xi)\big)\varphi(\xi)\|\\
&+\|(\xi|\xi|^{1+a})^{-1}\widetilde{\varphi}(\xi)\mathcal{D}^{(\theta_1-1)}\big(\varphi(\xi))\|.
\end{aligned}    
\end{equation*}
We follow with the estimates of each factor in the previous inequality. Using Proposition \ref{pontualn}, and that $|(\xi|\xi|^{1+a})^{-1}|\leq |\xi|^{-(2+a)}$, it is seen that
\begin{equation*}
\begin{aligned}
 \|\mathcal{D}^{(\theta_1-1)}\big((e^{it\xi|\xi|^{1+a}}-1)\big)(\xi|\xi|^{1+a})^{-1}\varphi(\xi)\|\les & \langle T \rangle\||\xi|^{(1+a)(\theta_1-1)-(2+a)}\varphi\|
 \les  \langle T \rangle,
\end{aligned}    
\end{equation*}
where we have used that $|\xi|^{(1+a)(\theta_1-1)-(2+a)}\varphi\in L^2(\mathbb{R})$ provided that our conditions assure $(1+a)(\theta_1-1)-(2+a)+\frac{1}{2}>0$. Given that $\theta_1-1<-(2+a)+\frac{1}{2}$, the results in \cite[Proposition 2.9]{FLP1} establish that $\mathcal{D}^{(\theta_1-1)}\big(\sgn(\xi)|\xi|^{-(2+a)}\widetilde{\varphi}(\xi)\big)\varphi(\xi)\in L^2(\mathbb{R})$, thus we get
\begin{equation*}
\begin{aligned}
\|\mathcal{D}^{(\theta_1-1)}\big((\xi|\xi|^{1+a})^{-1}\widetilde{\varphi}(\xi)\big)\varphi(\xi)\| =\|\mathcal{D}^{(\theta_1-1)}\big(\sgn(\xi)|\xi|^{-(2+a)}\widetilde{\varphi}(\xi)\big)\varphi(\xi)\| \les 1.
\end{aligned}    
\end{equation*}
Since $-(2+a)>0$ and $\varphi, \widetilde{\varphi}\in C^{\infty}_0(\mathbb{R})$, we obtain
\begin{equation*}
\begin{aligned}
\|(\xi|\xi|^{1+a})^{-1}\widetilde{\varphi}(\xi)\mathcal{D}^{(\theta_1-1)}\big(\varphi(\xi))\|\les &  \||\xi|^{-(2+a)}\widetilde{\varphi}\|_{L^{\infty}}\|\mathcal{D}^{(\theta_1-1)}\varphi\|
\les  1.  
\end{aligned}    
\end{equation*}
Given that $1<\theta_1<\frac{1+2a}{2(1+a)}$, we can apply Lemma \ref{Dthetau}, and familiar arguments to conclude $\mathcal{D}^{(\theta_1-1)}\big(e^{it\xi|\xi|^{1+a}}\varphi) \in L^2(\mathbb{R})$. This establishes \eqref{thingstoprove2}.

Finally, we remark that the arguments in the last case hold for $\theta=\Big(-\frac{3}{2(1+a)}\Big)^{-}=-\frac{3}{2(1+a)}-\epsilon$, and $\theta_1=-\frac{3}{2(1+a)}$ provided that $s\geq \max\{-\frac{9+6a}{2(5+2a)},2^{+}\}$, and $0<\epsilon \ll 1$. This completes the proof of Proposition \ref{extradecaycond}. 
\end{proof}

We are now in a position to prove Theorem \ref{lwpw} (iii) case $-\frac{3}{2(1+a)}\leq \theta<\frac{1+2a}{2(1+a)}$, $s\geq \max\{-\frac{9+6a}{2(5+2a)},2^{+}\}$. Consider $\phi\in \mathcal{Z}_{s,\theta}^{a,\theta}$. By the results proved in the previous case, i.e.,  Theorem \ref{lwpw} (i) with weights $<\frac{-3}{2(1+a)}$, there exist a time and a unique solution $u$ of \eqref{gbo} in the class
\begin{equation}\label{iterationclass1}
  u\in C(0,T];H^s_{a,\theta_0}(\mathbb{R}))\cap L^{\infty}([0,T];L^2(|x|^{2\theta_0}\, dx)),
\end{equation}
where $\theta_0=\Big(-\frac{3}{2(1+a)}\Big)^{-}$, and we have also used Lemma \ref{existence} to assure that $u$ describes a continuous curve in the class $H^s_{a,\theta_0}(\mathbb{R})$. We also remark that our present assumptions on $s$ show that the conditions of regularity in the case weights $0<\theta<-\frac{3}{2(1+a)}$ hold. It follows from Lemma \ref{Dthetau} and hypothesis on $\phi$ that 
\begin{equation}\label{iterationclass2}
  U(t)\phi\in L^{\infty}([0,T];L^2(|x|^{2\theta}\, dx).  
\end{equation}
By Proposition \ref{extradecaycond} and \eqref{iterationclass1}, we have that 
\begin{equation*}
\int_{0}^t U(t-\tau)(\partial_x u^2)(\tau)\, d\tau\in L^{\infty}([0,T]; L^2(|x|^{2\theta_1}\, dx)),   
\end{equation*}  
with $\theta_1=-\frac{3}{2(1+a)}$. Then the previous remark and \eqref{iterationclass2} yield 
\begin{equation}\label{iterationclass3}
  u\in  C(0,T];H^s_{a,\theta_0}(\mathbb{R}))\cap L^{\infty}([0,T];L^2(|x|^{2\theta_1}\, dx)).
\end{equation}
If $\theta_1=\theta$, we are done. Otherwise, if $\theta_1<\theta$, we consider $\theta_2=\min\left\{\Big(\frac{2(s-1)+s}{2s}\Big)\theta_1,\theta\right\}$ (notice that $\Big(\frac{2(s-1)+s}{2s}\Big)\theta_1$ is the middle point of the interval $(\theta_1,\frac{2(s-1)}{s}\theta_1)$). It follows from Proposition \ref{extradecaycond} that
\begin{equation*}
\int_{0}^t U(t-\tau)(\partial_x u^2)(\tau)\, d\tau\in L^{\infty}([0,T]; L^2(|x|^{2\theta_2}\, dx)),    
\end{equation*}  
which together with \eqref{iterationclass2} show
\begin{equation*}
  u\in  C(0,T];H^s_{a,\theta_0}(\mathbb{R}))\cap L^{\infty}([0,T];L^2(|x|^{2\theta_2}\, dx)).
\end{equation*}
Hence, if $\theta_2=\theta$ we are done. Otherwise, we can iterate the previous argument with the sequence $\theta_j=\min\{\Big(\frac{2(s-1)+s}{2s}\Big)\theta_{j-1},\theta\}$, $j\in \mathbb{Z}^{+}$, $j\geq 2$ to obtain 
\begin{equation*}
  u\in  C(0,T];H^s_{a,\theta_0}(\mathbb{R}))\cap L^{\infty}([0,T];L^2(|x|^{2\theta_j}\, dx)).
\end{equation*}
However, $s\geq \max\{-\frac{9+6a}{2(5+2a)},2^{+}\}$ is independent of $\theta$, and $\frac{2(s-1)+s}{2s}>1$. These facts imply that there has to be a finite number of iterations $J\geq 2$ such that $\theta_j=\theta$ for all $j\geq J$. This completes the proof of the theorem.
\end{proof}


\section{Unique continuation principles}\label{uniquesect}

In this section, we deduce Theorem \ref{Uniquecont1}, and the unique continuation principles stated in Propositions \ref{propostimes1} and \ref{propostimes3}.

\begin{proof}[Proof of Theorem \ref{Uniquecont1}] Under our assumptions, let $u\in C([0,T];H^s(\mathbb{R}))\cap L^{\infty}([0,T];L^2(|x|^{2\theta}\, dx))$ be a solution of \eqref{gbo} with $u(0)=\phi\in H^s(\mathbb{R})$.

Let us start with part (i), where we assume $0<\theta\leq 1$. We will only deduce the most difficult case $0<\theta<1$ as $\theta=1$ follows from our ideas below, and using the local derivative $\frac{d}{d\xi}$ instead of the fractional derivative $\mathcal{D}^{\theta}$. By the decay assumption on $u$, and the integral formulation of \eqref{gbo}, we have
\begin{equation}\label{int}
u(t)=U(t)\phi -\frac12\int_{0}^{t}U(t-\tau)\partial_{x}u^2(\tau)d\tau\in L^{\infty}([0,T];L^2(|x|^{2\theta}\, dx)).
\end{equation}
Now, since $u\in C([0,T];H^s(\mathbb{R}))\cap L^{\infty}([0,T]; L^2(|x|^{2\theta}\, dx))$,  estimates in \eqref{des1Phi0}, \eqref{des1Phi}, \eqref{des1Phi2} and \eqref{des1Phi3} yield
\begin{equation}\label{integralequartiondecay}
\int_{0}^{t}U(t-\tau)\partial_{x}u^2(\tau)d\tau\in L^{\infty}([0,T];L^2(|x|^{2\theta}\, dx)).   
\end{equation}
Hence, it follows from \eqref{int} that
\begin{equation*}
U(t)\phi\in L^2(|x|^{2\theta}\, dx) \, \, \text{ for almost every } \, \, t\in [0,T],
\end{equation*}
which is equivalent to
\begin{equation}\label{conclusDecay}
D^{\theta}\Big(e^{it\xi|\xi|^{1+a}}\widehat{\phi}\Big)\in L^2(\mathbb{R}) \, \, \text{ for almost every } \, \, t\in [0,T].
\end{equation}
Let $t,t_1\in [0,T]$ be such that the above statement holds true, and $t-t_1>0$. We define $\widetilde{\phi}(\xi):=e^{it_1\xi|\xi|^{1+a}}\widehat{\phi}(\xi)$, i.e., $\widetilde{\phi}$ is the Fourier transform of $U(t_1)\phi$.  It follows from \eqref{conclusDecay} that $D^{\theta}(e^{i(t-t_1)\xi|\xi|^{1+a}}\widetilde{\phi}),\, \, D^{\theta}(\widetilde{\phi})\in L^2(\mathbb{R})$. Notice that we have considered two times $t$ and $t_1$ as at this point, we do not know if $D^{\theta}(\widehat{\phi})\in L^2(\mathbb{R})$.

Now, for almost every $\xi$ and $\eta$, we have
\begin{equation*}
\begin{aligned}
|\big(e^{i(t-t_1)\xi|\xi|^{1+a}}-e^{i(t-t_1)\eta|\eta|^{1+a}}\big)\widetilde{\phi}(\xi)|^2\les & |e^{i(t-t_1)\xi|\xi|^{1+a}}\widetilde{\phi}(\xi)-e^{i(t-t_1)\eta|\eta|^{1+a}}\widetilde{\phi}(\eta)|^2  \\
&+ |\widetilde{\phi}(\xi)-\widetilde{\phi}(\eta)|^2,
\end{aligned}    
\end{equation*}
then the above inequality and the definition of the fractional derivative $\mathcal{D}^{\theta}$ imply that for almost every $\xi$,
\begin{equation*}
 |\widetilde{\phi}(\xi)|\mathcal{D}^{\theta}(e^{i(t_1-t_2)\xi|\xi|^{1+a}})(\xi)\les \mathcal{D}^{\theta}(e^{i(t-t_1)\xi|\xi|^{1+a}}\widetilde{\phi})(\xi)+\mathcal{D}^{\theta}(\widetilde{\phi})(\xi).  
\end{equation*}
Given that $D^{\theta}(e^{i(t-t_1)\xi|\xi|^{1+a}}\widetilde{\phi}),\, \, D^{\theta}(\widetilde{\phi})\in L^2(\mathbb{R})$, by Theorem \ref{stein}, and the previous inequality, it must be the case that
\begin{equation}\label{newconclu1}
  |\widetilde{\phi}(\xi)|\mathcal{D}^{\theta}(e^{i(t-t_1)\xi|\xi|^{1+a}})(\xi)\in L^2(\mathbb{R}).   
\end{equation}
However, arguing as in the proof of Proposition \ref{pontualn}, we can find a constant $0<c<1$ for which 
\begin{equation}\label{newconclu2}
 \mathcal{D}^{\theta}(e^{i(t-t_1)\xi|\xi|^{1+a}})(\xi)\mathbbm{1}_{\{0<|\xi|\leq c\}}(\xi)\gtrsim |\xi|^{(1+a)\theta}\mathbbm{1}_{\{0<|\xi|\leq c\}}(\xi),  
\end{equation}
where $\mathbbm{1}_{\{0<|\xi|\leq c\}}$ denotes the indicator function on the set $\{\xi\in \mathbb{R}: 0<|\xi|\leq c\}$, and the implicit constant above is independent of $\xi$ in such a set and depends on $t-t_1>0$. Consequently, using that $|\widetilde{\phi}|=|\widehat{\phi}|$, together with \eqref{newconclu1} and \eqref{newconclu2}, we get $|\xi|^{(1+a)\theta}|\widehat{\phi}|\mathbbm{1}_{\{0<|\xi|\leq c\}}\in L^2(\mathbb{R})$. But since $\phi\in L^2(\mathbb{R})$, and $(1+a)\theta<0$, it follows that $|\xi|^{(1+a)\theta}\widehat{\phi}\in L^2(\mathbb{R})$, i.e., $\phi\in \dot{H}^{(1+a)\theta}(\mathbb{R})$.

Now, we can prove that it holds $D^{\theta}(\widehat{\phi})\in L^2(\mathbb{R})$, that is to say, $\phi\in L^2(|x|^{2\theta}\, dx)$. Since for almost every $\xi$ and $\eta$,
\begin{equation*}
\begin{aligned}
|\widehat{\phi}(\xi)-\widehat{\phi}(\eta)|^{2}\les |e^{i t_1\xi|\xi|^{1+a}}-e^{i t_1\eta|\eta|^{1+a}}|^2|\widehat{\phi}(\xi)|^2+|e^{i t_1\xi|\xi|^{1+a}}\widehat{\phi}(\xi)-e^{i t_1\eta|\eta|^{1+a}}\widehat{\phi}(\eta)|^2,     
\end{aligned}   
\end{equation*}
we get from the definition of $\mathcal{D}^{\theta}$ and Proposition \ref{pontualn} that
\begin{equation*}
\begin{aligned}
  \mathcal{D}^{\theta}(\widehat{\phi})(\xi)\les & \mathcal{D}^{\theta}(e^{it_1\xi|\xi|^{1+a}})(\xi)|\widehat{\phi}(\xi)|+\mathcal{D}^{\theta}(e^{it_1\xi|\xi|^{1+a}}\widehat{\phi})(\xi)\\
  \les & |\xi|^{(1+a)\theta}|\widehat{\phi}(\xi)|+\mathcal{D}^{\theta}(e^{it_1\xi|\xi|^{1+a}}\widehat{\phi})(\xi),
\end{aligned}
\end{equation*}
 for almost every $\xi$. The previous inequality, the fact that $\phi\in \dot{H}^{(1+a)\theta}(\mathbb{R})$, $U(t_1)\phi \in L^2(|x|^{2\theta}\, dx)$, and Theorem \ref{stein} imply that $D^{\theta}(\widehat{\phi})\in L^2(\mathbb{R})$ as desired.

Finally, as we have already shown that $\phi\in \dot{H}^{(1+a)\theta}(\mathbb{R})\cap  L^2(|x|^{2\theta}\, dx)$, using that $u$ is a solution of the integral equation associated to \eqref{gbo}, we can apply the ideas in Remark \ref{remarkona} to conclude that $u\in L^{\infty}([0,T];\dot{H}^{(1+a)\theta}(\mathbb{R}))$. This completes the proof of part (i) in Theorem \ref{Uniquecont1}.

Next we prove part (ii). Here we have $1<\theta<\frac{1+2a}{2(1+a)}$. We begin by noting that $\theta>1$ implies that $u\in L^{\infty}([0,T];L^2(|x|^2\, dx))$. Then an application of part (i) gives
\begin{equation}\label{conclusion1}
\begin{aligned}
    D^{1+a}\phi, \, \, x\phi \in L^2(\mathbb{R}).
\end{aligned}    
\end{equation}
Arguing as in the proof of part (i) above when $0<\theta<-\frac{3}{2(1+a)}$, and using Proposition \ref{extradecaycond} when $-\frac{3}{2(1+a)}\leq \theta<\frac{1+2a}{2(1+a)}$, it follows from \eqref{int} and \eqref{integralequartiondecay} that \eqref{conclusDecay} holds for $1<\theta<\frac{1+2a}{2(1+a)}$. However, such fact is equivalent to
\begin{equation}\label{decayextraunique1}
xU(t)\phi\in L^2(|x|^{2(\theta-1)}\, dx) \, \, \text{ for almost every } \, \, t\in [0,T].
\end{equation}
We emphasize that $0<\theta-1<1$ provided that $\theta<\frac{1+2a}{2(1+a)}$ and $a<-2$. We observe that the validity of \eqref{conclusion1} implies
\begin{equation*}
xU(t)\phi=U(t)((x\phi)-(2+a)tD^{1+a}\phi)\in L^2(\mathbb{R}),   
\end{equation*}
and thus \eqref{decayextraunique1} is equivalent to
\begin{equation}\label{equivalentstatement}
U(t)((x\phi)-(2+a)tD^{1+a}\phi)\in L^2(|x|^{2(\theta-1)}\, dx) \, \, \text{ for almost every } \, \, t\in [0,T].
\end{equation}
Now, let us assume that
\begin{equation}\label{firstassumpt}
D^{(1+a)(\theta-1)}(x\phi),\, \, |x|^{\theta}\phi  \in L^2(\mathbb{R}).
\end{equation}
It follows from Lemma \ref{Dthetau} that $U(t)(x\phi)\in L^2(|x|^{2(\theta-1)}\, dx)$. Hence, \eqref{equivalentstatement} implies that there exist two different times $t, t_1\in [0,T]$ such that
\begin{equation*}
  U(t)(D^{1+a}\phi),U(t_1)(D^{1+a}\phi)\in L^2(|x|^{2(\theta-1)}\, dx).  
\end{equation*}
Accordingly, we can use the same arguments in the proof of Theorem \ref{Uniquecont1} (i) with the function $D^{1+a}\phi$, and weight size $\theta-1$ to deduce
\begin{equation}\label{firstassumpt2}
   D^{(1+a)\theta}\phi,\, \, |x|^{(\theta-1)} D^{1+a}\phi\in L^2(\mathbb{R}).
\end{equation}
Conversely, if we assume that the conditions in \eqref{firstassumpt2} hold, we have from Lemma \ref{Dthetau} that $U(t)(D^{1+a}\phi)\in L^2(|x|^{2(\theta-1)}\, dx)$. Then \eqref{equivalentstatement} yields that there exist two different times $t, t_1\in [0,T]$ such that
\begin{equation*}
  U(t)(x\phi),U(t_1)(x\phi)\in L^2(|x|^{2(\theta-1)}\, dx).  
\end{equation*}
It follows from the arguments in the proof of part (i)  that \eqref{firstassumpt} holds.  
\end{proof}

\begin{remark}
We note that in the proof of Theorem \ref{Uniquecont1}, the linear part of the equation mainly establishes the conclusion in this theorem. This can be justified by the fact that we used and deduced the following result.
\begin{proposition}\label{proptwotimes}
Let $\phi \in H^{s}(\mathbb{R})$, $s\geq 0$ be such that for two different times $t_1,t_2\in \mathbb{R}$, one has
\begin{equation*}
    U(t_j)\phi\in L^2(|x|^{2\theta}\, dx),
\end{equation*}    
$j=1,2$.
\begin{itemize}
    \item[(i)] If $0<\theta\leq 1$, it follows that 
    \begin{equation*}
     D^{(1+a)\theta}\phi, \, |x|^{\theta}\phi\in L^2(\mathbb{R}).   
    \end{equation*}
\item[(ii)] If $1<\theta\leq 2$. Then we have
\begin{equation*}
U(t_j)((x\phi)-(2+a)t_jD^{1+a}\phi)\in L^2(|x|^{2(\theta-1)}\, dx),    
\end{equation*}
$j=1,2$. In particular, it follows that \eqref{firstassumpt} holds if and only if \eqref{firstassumpt2} holds.
\end{itemize}
\end{proposition}
\end{remark}

We emphasize that Proposition \ref{proptwotimes} (ii) holds for the range $1<\theta\leq 2$. In contrast, Theorem \ref{Uniquecont1} (ii) is valid for $1<\theta<\frac{1+2a}{2(1+a)}$. This is a consequence of the estimates for the nonlinear term $u\partial_x u$ in the deduction of Theorem \ref{lwpw}.


\subsection{Proof of Proposition \ref{propostimes1}}
With Propositions \ref{extradecaycond} and \ref{proptwotimes} in hand we are able to establish  Proposition \ref{propostimes1}.

\begin{proof}[Proof of Proposition \ref{propostimes1}] Let $u\in C([0,T];Z_{s,1}^{a,1})$ be a solution of \eqref{gbo} with initial condition $u(0)=\phi$. Since there exist two different times $t_1,t_2\in [0,T]$ such that $ u(t_j)\in L^2(|x|^{2^{+}}\, dx)$, $j=1,2$, and Proposition \ref{extradecaycond} establishes extra decay for the integral factor $u(t)-U(t)\phi$, it follows from the integral formulation of \eqref{gbo} (see \eqref{int}) that
\begin{equation*}
    U(t_j)\phi \in L^2(|x|^{2^{+}}\,dx),
\end{equation*}
for each $j=1,2$. Consequently, we can apply Proposition \ref{proptwotimes} to get the desired result.
\end{proof}


\subsection{Proof of Proposition \ref{propostimes3}}

We first present the following consequence of the proof of Proposition \ref{extradecaycond}.

\begin{remark}\label{uniquecontremark}
Some estimates in the proof of Proposition \ref{extradecaycond} extend to weights of order $\frac{1+2a}{2(1+a)}$. More precisely, let $s\geq \max\{-\frac{9+6a}{2(5+2a)},2^{+}\}$, and $\theta=\Big(\frac{1+2a}{2(1+a)}\Big)^{-}$. Assume that $u\in C([0,T];H^s(\mathbb{R}))\cap L^{\infty}([0,T];L^2(|x|^{2\theta}\, dx))$. The deduction of \eqref{thingstoprove} in the proof of Proposition \ref{extradecaycond} implies
\begin{equation}\label{equniquecontremark}
   \int_0^{t} U(t-\tau)(x\partial_x u^2)(\tau) d\tau\in  L^{\infty}([0,T];L^2(|x|^{2\big(\frac{1+2a}{2(1+a)}-1\big)}\, dx)). 
\end{equation}    
Moreover, under the same conditions on $u$, and recalling the factors $\mathcal{A}_2$, $\mathcal{A}_3$ in \eqref{decompositionApart}, and $\mathcal{A}_{1,2}$ in \eqref{decompositionApart2}. The arguments in the proof of Proposition \ref{extradecaycond} yield
\begin{equation}\label{equniquecontremark2}
   D^{\frac{1+2a}{2(1+a)}-1}\mathcal{A}_{1,2},\, D^{\frac{1+2a}{2(1+a)}-1}\mathcal{A}_2,\,  D^{\frac{1+2a}{2(1+a)}-1}\mathcal{A}_3 \in  L^{\infty}([0,T];L^2(\mathbb{R})). 
\end{equation}  
\end{remark}

\begin{proof}[Proof of Proposition \ref{propostimes3}]
Without loss of generality, we will assume that $t_1=0$. Thus, denoting $u(0)=\phi$, we have that the assumptions in  Proposition \ref{propostimes3} are 
\begin{equation*}
u(t_2), \, \phi\in L^2(|x|^{2\big(\frac{1+2a}{2(1+a)}\big)}\, dx), 
\end{equation*}
and
 \begin{equation*}
D^{(1+a)\big(\frac{1+2a}{2(1+a)}\big)}\phi,\, \, D^{(1+a)\big(\frac{1+2a}{2(1+a)}-1\big)}(x \phi),\, \, |x|^{\frac{1+2a}{2(1+a)}-1}D^{1+a}\phi\in L^2(\mathbb{R}).        
\end{equation*}
The previous conditions and Lemma \ref{Dthetau} imply that $U(t)\phi\in L^{\infty}([0,T];L^2(|x|^{2\big(\frac{1+2a}{2(1+a)}\big)}\, dx))$. Consequently, we have that $u(t_2)\in L^2(|x|^{2\big(\frac{1+2a}{2(1+a)}\big)}\, dx)$ if and only if
\begin{equation*}
   \int_0^{t_2} U(t_2-\tau)(\partial_x u^2)(\tau)\, d\tau\in  L^2(|x|^{2\big(\frac{1+2a}{2(1+a)}\big)}\, dx), 
\end{equation*}
which in view of the identity $xU(t) f=U(t)(xf-(2+a)tD^{1+a}f)$ is equivalent to
\begin{equation*}
   \int_0^{t_2} U(t_2-\tau)(x\partial_x u^2-(2+a)(t_2-\tau)\partial_x D^{1+a}u^2)(\tau)\, d\tau\in  L^2(|x|^{2\big(\frac{1+2a}{2(1+a)}-1\big)}\, dx). 
\end{equation*}
Hence, from \eqref{equniquecontremark} in Remark \ref{uniquecontremark}, and the preceding conclusion it follow
\begin{equation}\label{finaleq1}
   \int_0^{t_2} U(t_2-\tau)((t_2-\tau)\partial_x D^{1+a}u^2)(\tau)\, d\tau\in  L^2(|x|^{2\big(\frac{1+2a}{2(1+a)}-1\big)}\, dx). 
\end{equation}
Now, going to the frequency domain and using the decomposition \eqref{decompositionApart} with the same terms $\mathcal{A}_1$, $\mathcal{A}_2$ and $\mathcal{A}_3$ (also recall the decomposition of $\mathcal{A}_1$ in \eqref{decompositionApart2}), we use \eqref{equniquecontremark2} to deduce that \eqref{finaleq1} is equivalent to 
\begin{equation}\label{finaleq2}
 D^{\big(\frac{1+2a}{2(1+a)}-1\big)}\big(e^{it_2\xi|\xi|^{1+a}}\varphi(\xi)\big)\|\phi\|^2 \in L^2(\mathbb{R}), 
\end{equation}
which follows if and only if $\mathcal{D}^{\big(\frac{1+2a}{2(1+a)}-1\big)}\big(e^{it_2\xi|\xi|^{1+a}}\varphi(\xi)\big)\|\phi\|^2\in L^2(\mathbb{R})$, where $\varphi\in C^{\infty}_0(\mathbb{R})$ with $0\leq \varphi \leq 1$, $\varphi(\xi)=1$, whenever $|\xi|\leq 1$. However, since $\frac{1+2a}{2(1+a)}-1=\frac{-1}{2(1+a)}$, the result of Proposition \ref{nonintegrabprop} shows
\begin{equation*}
 \mathcal{D}^{\big(\frac{1+2a}{2(1+a)}-1\big)}\big(e^{it_2\xi|\xi|^{1+a}}\varphi(\xi)\big) \notin L^2(\mathbb{R}),   
\end{equation*}
but then \eqref{finaleq2}  forces us to have $\|\phi\|=0$. The proof is thus completed.
\end{proof}


\section{Appendix}\label{standarlwpAppendix}

This part is devoted to deriving the local well-posedness result presented in Lemma \ref{existence}. We will use a parabolic regularization argument, as given in \cite{AbBonaFellSaut1989,APlow,Iorio1986,IorioNunes1998}. We emphasize that the presence of the operator $\partial_x D^{1+a}$ complicates certain estimates and motivates to detail the proof of Lemma \ref{existence}.

\begin{proof}[Proof of Lemma \ref{existence}] We will only deduce existence and uniqueness of solutions for \eqref{gbo} and \eqref{mugbo}. The continuous dependence follows from our arguments and the ideas in  references \cite{AbBonaFellSaut1989,APlow,Iorio1986,IorioNunes1998}. To simplify the presentation of our estimates, we will divide our considerations into five key steps.

\underline{\bf (i)} Existence of solutions for the regularized equation \eqref{mugbo}. Let $0<\mu<1$ be fixed. Using the integral formulation of \eqref{mugbo}, one can apply a contraction argument similar to the one developed in the proof of Theorem \ref{lwpw} to show that there exist a time $T_{\mu}>0$ and a unique solution $u_{\mu}$ of \eqref{mugbo} with initial condition $u_{\mu}(0)=\phi$ such that
    \begin{equation*}
        u_{\mu}\in C([0,T_{\mu}];H^s_{a,\theta}(\mathbb{R})).
    \end{equation*}
The above result depends strongly on the estimates in Remark \ref{remarkona} for the nonlinear term $u_{\mu}\partial_x u_{\mu}$. 

\underline{\bf (ii)} We will show that the solutions $u_{\mu}$, $0<\mu<1$ constructed in (i) can be extended to a common time $T>0$ independent of $\mu$ and depending on $\phi$, and there exists $\rho\in C([0,T];[0,\infty))$, independent of $\mu$, such that
\begin{equation}\label{uniboundrho}
\begin{aligned}
 \|u_{\mu}(t)\|_{H^s_{a,\theta}}^2\leq \rho(t).
\end{aligned}
\end{equation}
We first observe that using energy estimates with the equation \eqref{mugbo}, and the embedding $H^{\frac{3}{2}}(\mathbb{R}) \hookrightarrow W^{1,\infty}(\mathbb{R})$, one gets the differential inequality
\begin{equation}\label{differentialineqA1}
  \frac{d}{dt}\|u_{\mu}(t)\|^2_{H^s}\leq c_s \|u_{\mu}(t)\|_{H^s}^3,  
\end{equation}
where the constant $c_s>0$ is independent of $\mu$. It is important to mention that the energy estimate used to get the above expression, and the fact that the term involving the dispersion can be canceled, are justified by the condition $\theta\geq \frac{2+a}{1+a}>\frac{1}{2}+\frac{1}{2(1+a)}$. To see this, writing $\partial_x=\mathcal{H}D$, where $\mathcal{H}$ denotes the Hilbert transform, it follows  
\begin{equation}\label{parabolicArgument}
\int \partial_x D^{1+a}J^s u_{\mu }J^s u_{\mu}\, dx=\int \mathcal{H}D^{\frac{2+a}{2}}J^s u_{\mu}D^{\frac{2+a}{2}}J^s u_{\mu} \, dx.    
\end{equation}
Thus, since $\mathcal{H}$ defines a skew-symmetric bounded operator in $L^2(\mathbb{R})$, the above integral is zero if we prove $D^{\frac{2+a}{2}}J^s u_{\mu}\in L^2(\mathbb{R})$. We note that in \eqref{parabolicArgument}, we do not use directly that the operator $\partial_x D^{1+a}$ is skew-symmetric, as we first need to justify that such integral is defined. Now, Plancherel's identity, the fact that $\theta>\frac{1}{2}+\frac{1}{2(1+a)}$, and dividing into frequencies, we observe
\begin{equation*}
\begin{aligned}
\|D^{\frac{2+a}{2}}J^s u_{\mu}\|\leq & \||\xi|^{\frac{2+a}{2}-(1+a)\theta}\langle \xi \rangle^s|\xi|^{(1+a)\theta}\widehat{u_{\mu}}\|_{L^2(|\xi|\leq 1)}+ \||\xi|^{\frac{2+a}{2}}\langle \xi \rangle^s\widehat{u_{\mu}}\|_{L^2(|\xi|\geq 1)}\\
\les & \|D^{(1+a)\theta}u_{\mu}\|+\|u_{\mu}\|_{H^s}.
\end{aligned}    
\end{equation*}
Consequently, the above estimate implies that \eqref{parabolicArgument} is null. Next, let $\rho_1\in C([0,T^{\ast});[0,\infty))$ be the maximal extended solution of
\begin{equation*}
\left\{\begin{aligned}
&\partial_t \rho_1(t)=c_s \rho_1(t)^{\frac{3}{2}}, \\
&\rho_1(0)=\|\phi\|_{H^s}^2.
\end{aligned}
\right.
\end{equation*}
Then one has from \eqref{differentialineqA1} that 
\begin{equation*}
  \|u_{\mu}(t)\|_{H^s}^2\leq \rho_1(t). 
\end{equation*}
Using the integral formulation of \eqref{mugbo}, the fact that $\{U_{\mu}(t)\}$ defines a semigroup of contractions (see \eqref{musemi}), \eqref{eqparabolicarg1}, and \eqref{eqparabolicarg2}, we have
\begin{equation*}
\begin{aligned}
\|D^{(1+a)\theta}u_{\mu}(t)\|\leq & \|D^{(1+a)\theta}\phi\|+\frac{1}{2}\int_0^t\|D^{(1+a)\theta}\partial_x(u^2_{\mu})(\tau)\|\, d\tau  \\
\leq & \|D^{(1+a)\theta}\phi\|+c\int_0^t\|u_{\mu}(\tau)\|_{H^s}^2\, d\tau \\
\leq & \|D^{(1+a)\theta}\phi\|+c \int_0^t\rho_1(\tau)\, d\tau.
\end{aligned}    
\end{equation*}
Summarizing, we have deduced that
\begin{equation}\label{parabolicUniformbound}
\begin{aligned}
 \|u_{\mu}(t)\|_{H^s_{a,\theta}}^2=& \|u_{\mu}(t)\|^2_{H^s}+\|D^{(1+a)\theta}u_{\mu}(t)\|^2 \\
 \leq&\rho_1(t)+\Big(\|D^{(1+a)\theta}\phi\|+c\int_0^t\rho_1(\tau)\, d\tau\Big)^{2}=:\rho(t).  
\end{aligned}
\end{equation}
Consequently, since $\rho \in C([0,T^{\ast});[0,\infty))$, and $T^{\ast}$ do not depend on $\mu$, it follows that for any time $0<T<T^{\ast}=T^{\ast}(\phi)$, the inequality \eqref{parabolicUniformbound}, and the usual extension method show that $u_{\mu}$ can be extended to the class $C([0,T];H^s_{a,\theta}(\mathbb{R}))$.

\underline{\bf (iii)} By the previous step, let $T>0$ be such that $u_{\mu}\in C([0,T];H^s_{a,\theta}(\mathbb{R}))$ for all $0<\mu<1$. In this part, we will show that when $\mu\to 0^{+}$, $u_{\mu}$ converges to a function $u$ in $C([0,T];L^2(\mathbb{R}))$, and in the weak sense of $C_{w}([0,T];H^s_{a,\theta}(\mathbb{R}))$. Moreover, such limit $u$ satisfies \eqref{uniboundrho}. 

Using the equation in \eqref{mugbo}, and the cancellation of the dispersive term provided by the argument below \eqref{parabolicArgument}, one can use energy estimates to get that there exists $u\in C([0,T];L^2(\mathbb{R}))$ such that $u_{\mu}\to u$ as $\mu \to 0^{+}$ in the sense of  $C([0,T];L^2(\mathbb{R}))$. Moreover, it is not hard to see that this fact implies that $u_{\mu}\to u$ in the weak sense of $C_{w}([0,T];H^s(\mathbb{R}))$.

Now, we claim that $u_{\mu}\to u$ in $C_{w}([0,T];\dot{H}^{(1+a)\theta}(\mathbb{R}))$. But first, we will introduce some approximations for an arbitrary function $f\in \dot{H}^{(1+a)\theta}(\mathbb{R})$. We consider $\psi\in C^{\infty}(\mathbb{R})$ be such that $\psi(\xi)=1$, if $|\xi|\geq 2$, and $\psi(\xi)=0$, if $|\xi|\leq 1$. We define then $\psi_n(\xi)=\psi(n\xi)$ and  $f_n=(\psi_n\widehat{f})^{\vee}$, for each integer $n\geq 2$.  We observe that the support of the function $\psi_n$ implies that for almost every $\xi$,
 \begin{equation*}
||\xi|^{(1+a)\theta}\widehat{f}(\xi)-|\xi|^{(1+a)\theta}\psi_n(\xi)\widehat{f}(\xi)|^2=|1-\psi_n(\xi)|^2||\xi|^{(1+a)\theta}\widehat{f}(\xi)|^2\to 0   
 \end{equation*}
as $n\to \infty$, and
 \begin{equation*}
 \begin{aligned}
 ||\xi|^{(1+a)\theta}\widehat{f}(\xi)-|\xi|^{(1+a)\theta}\psi_n(\xi)\widehat{f}(\xi)|^2 \les  ||\xi|^{(1+a)\theta}\widehat{f}(\xi)|^2\in L^1(\mathbb{R}).
 \end{aligned}
 \end{equation*}
Consequently, the previous facts and Lebesgue's dominated convergence theorem yield $f_n\to f$ as $n\to \infty$ in $\dot{H}^{(1+a)\theta}(\mathbb{R})$. Additionally, 
\begin{equation*}
  \|D^{2(1+a)\theta}f_n\|\les n^{-(1+a)\theta}\|D^{(1+a)\theta}f\|.
\end{equation*}
Now, denoting by $\langle \cdot, \cdot \rangle_{\dot{H}^{(1+a)\theta}}$ the usual inner product of $\dot{H}^{(1+a)\theta}(\mathbb{R})$ (recall that $\dot{H}^{\beta}(\mathbb{R})$ is a Hilbert space if $\beta<\frac{1}{2}$), we have for $\mu,\nu \in (0,1)$
\begin{equation}\label{innerporductestimate}
\begin{aligned}
|\langle u_{\mu}(t)&-u_{\nu}(t), f\rangle_{\dot{H}^{(1+a)\theta}}|\\
\leq &|\langle u_{\mu}(t)-u_{\nu}(t), f-f_n\rangle_{\dot{H}^{(1+a)\theta}}|+|\langle u_{\mu}-u_{\nu}, f_n\rangle_{\dot{H}^{(1+a)\theta}}| \\
\leq &\|u_{\mu}(t)-u_{\nu}(t)\|_{\dot{H}^{(1+a)\theta}}\|f-f_n\|_{\dot{H}^{(1+a)\theta}}+|\langle u_{\mu}-u_{\nu}, D^{2(1+a)\theta}f_n\rangle_{L^2}|\\
\les &M\|f-f_n\|_{\dot{H}^{(1+a)\theta}}+n^{-(1+a)\theta}\Big(\sup_{t\in [0,T]}\| u_{\mu}(t)-u_{\nu}(t)\|\Big)\| D^{(1+a)\theta}f_n\|,
\end{aligned}    
\end{equation}
where using \eqref{parabolicUniformbound}, $M>0$ is such that $\sup\limits_{0<\mu<1}\sup\limits_{t\in [0,T]}\|u_{\mu}(t)\|_{H^s_{a,\theta}}\leq M$. Given that $u_{\mu}$ converges in $C([0,T];L^2(\mathbb{R}))$ and that $f_n\to f$ in $\dot{H}^{(1+a)\theta}(\mathbb{R})$,  inequality \eqref{innerporductestimate} establishes that $u_{\mu}$ converges in $C_{w}([0,T];\dot{H}^{(1+a)\theta}(\mathbb{R}))$, which leads to $u\in C_{w}([0,T];\dot{H}^{(1+a)\theta}(\mathbb{R}))$. Consequently, we have proved that when $\mu \to 0^{+}$, $u_{\mu}$ converges to $u$ in $C_{w}([0,T];H^s_{a,\theta}(\mathbb{R}))$. Using this fact, we can also deduce that $u$ satisfies \eqref{uniboundrho}.

\underline{\bf (iv)} Let $u\in C([0,T];L^2(\mathbb{R}))\cap C_{w}([0,T];H^s_{a,\theta}(\mathbb{R}))$ be the weak limit of the family $u_{\mu}$ of solutions of \eqref{mugbo} deduced in (iii), which also satisfies \eqref{uniboundrho}. We will show that $u\in C^1_w([0,T];H^{s-2}(\mathbb{R}))\cap AC[0,T];H^{s-2}(\mathbb{R}))$, $u(0)=\phi$, and $u$ satisfies \eqref{gbo} in a week sense, that is,
\begin{equation*}
\partial_t \langle u(t),f\rangle_{H^{s-2}}=\langle \partial_xD^{a+1}u(t)-u(t)\partial_x u(t),f\rangle_{H^{s-2}},  
\end{equation*}
for all $f\in H^{s-2}(\mathbb{R})$, and almost every $t\in [0,T]$, where $AC([0,T];H^{s-2}(\mathbb{R}))$ is the space of absolutely continuous functions with values in $H^{s-2}(\mathbb{R})$. Moreover, for $\phi\in H^{s}_{a,\theta}(\mathbb{R})$, the weak solution of \eqref{gbo} constructed above is unique in the class of bounded functions $u:[0,T]\rightarrow H^s_{a,\theta}(\mathbb{R})$, which are also in $C([0,T];L^2(\mathbb{R}))\cap C_w([0,T];H^s_{a,\theta}(\mathbb{R}))\cap C^1_w([0,T];H^{s-2}(\mathbb{R}))\cap AC[0,T];H^{s-2}(\mathbb{R}))$. 

To deduce (iv), we first observe that $\frac{2+a}{(1+a)}\leq \theta<-\frac{3}{2(1+a)}$, the fact that $u_\mu$ converges to $u$ in the sense of $ C_{w}([0,T];H_{a,\theta}^{s}(\mathbb{R}))$, and that $u$ satisfies \eqref{uniboundrho} imply
\begin{equation}\label{eqconvergenc}
\begin{aligned}
\partial_x D^{1+a}u_{\mu}  &\rightharpoonup  \partial_x D^{1+a}u, \, \,  \text{in }\, \,  H^{s-2}(\mathbb{R}), \\
u_{\mu}\partial_x u_{\mu}  &\rightharpoonup  u\partial_x u, \, \, \text{in }\, \,   H^{s-2}(\mathbb{R}),\\
\partial_x^2u_{\mu}  &\rightharpoonup  \partial_x^2u, \,  \, \text{in }\, \,   H^{s-2}(\mathbb{R}),
\end{aligned}    
\end{equation}
as $\mu\to 0^{+}$. The last two convergences above follow as a consequence of the results in (iii). However, we emphasize that the first convergence in \eqref{eqconvergenc} requires the condition $\frac{2+a}{(1+a)}\leq \theta<-\frac{3}{2(1+a)}$. To see this, let $f\in H^{s-2}(\mathbb{R})$ and $\varphi\in C^{\infty}_0(\mathbb{R})$ be such that $\varphi\equiv 1$, whenever $|\xi|\leq 1$. We set $P_{\varphi}f=(\varphi \widehat{f})^{\vee}$, then the definition of the inner product in $H^{s-2}(\mathbb{R})$ and the fact that $f=P_{\varphi}f+(I-P_{\varphi})f$ yield
\begin{equation*}
\begin{aligned}
\langle \partial_x D^{1+a}u_{\mu}-\partial_x D^{1+a}u, f \rangle_{H^{s-2}}=&\langle \partial_x D^{1+a-(1+a)\theta}J^{s-2} D^{(1+a)\theta}(u_{\mu}-u), J^{s-2}P_{\varphi}f \rangle_{L^2}\\
&+\langle \partial_x D^{1+a}J^{s-2}(u_{\mu}-u),(I-P_{\varphi})J^{s-2}f \rangle_{L^2}\\
=&-\langle  D^{(1+a)\theta}(u_{\mu}-u), \partial_x D^{1+a-(1+a)\theta}J^{2s-4}P_{\varphi}f \rangle_{L^2}\\
&-\langle J^{s-2}(u_{\mu}-u),(I-P_{\varphi})\partial_x D^{1+a}J^{s-2}f \rangle_{L^2}\\
=:& \mathcal{I}_{1,\mu}+\mathcal{I}_{2,\mu}.
\end{aligned}
\end{equation*}
Now, using the definition of the operator $P_{\varphi}$, and the fact that $2+a-(1+a)\theta\geq 0$, we get $\partial_x D^{1+a-(1+a)\theta}J^{2s-4}P_{\varphi}f \in L^2(\mathbb{R})$. We can prove using the week convergence of $u_{\mu}$ to $u$ deduced in (iii), and \eqref{uniboundrho}, that $D^{(1+a)\theta}u_{\mu}\rightharpoonup D^{(1+a)\theta}u$ in the $L^2(\mathbb{R})$ sense. Collecting these previous facts, we have that $\mathcal{I}_{1,\mu}\to 0$ as $\mu \to 0^{+}$. On the other hand, by support properties of the operator $I-P_{\varphi}$, and using that $a<-2$, we have $(I-P_{\varphi})\partial_x D^{1+a}J^{s-2}f \in L^2(\mathbb{R})$. Thus, the week convergence of $u_{\mu}$ in $H^s(\mathbb{R})$ implies that $\mathcal{I}_{2,\mu}\to 0$ as $\mu\to 0^+$. This shows the convergences in \eqref{eqconvergenc}. 

Now, given that $u_{\mu}$ solves the integral equation associated to \eqref{mugbo}, it follows that for any $f\in H^{s-2}(\mathbb{R})$,
\begin{equation*}
\begin{aligned}
\langle u_{\mu}(t), f \rangle_{H^{s-2}}=\langle \phi, f \rangle_{H^{s-2}}+\int_0^t \langle \mu \partial_x^2 u_{\mu}(\tau)+ \partial_x D^{a}u_{\mu}(\tau)-u_{\mu}(\tau)\partial_x u_{\mu}(\tau), f \rangle_{H^{s-2}} \, d\tau.   
\end{aligned}    
\end{equation*}
Hence, \eqref{eqconvergenc} shows that when $\mu \to 0^{+}$, 
\begin{equation*}
\begin{aligned}
\langle u(t), f \rangle_{H^{s-2}}=\langle \phi, f \rangle_{H^{s-2}}+\int_0^t \langle  \partial_x D^{a}u(\tau)-u(\tau)\partial_x u(\tau), f \rangle_{H^{s-2}} \, d\tau.   
\end{aligned}    
\end{equation*}
The above identity and similar ideas to those in \cite{AbBonaFellSaut1989,APlow,Iorio1986,IorioNunes1998}, show that $u$ satisfies the equation in \eqref{gbo} in the weak sense of $H^{s-2}(\mathbb{R})$ and $u\in AC[0,T];H^{s-2}(\mathbb{R}))$. Finally, using energy estimates at $L^2$-level, one gets the uniqueness statement.

\underline{\bf (v)} Let $u$ be the solution obtained in (iv), we will show $u\in C([0,T];H^s_{a,\theta}(\mathbb{R}))$. Note that once we have proved that $u\in C([0,T];H^s_{a,\theta}(\mathbb{R}))$, the conclusion in (iv) completes the proof of existence of solutions for the equation \eqref{gbo}. 
Since $u\in C_w([0,T];H^s_{a,\theta}(\mathbb{R}))$ is a solution of \eqref{gbo}, and $u$ satisfies \eqref{uniboundrho}, given $f\in H^{s}_{a,\theta}(\mathbb{R})$ with $\|f\|_{H^{s}_{a,\theta}}=1$, we get
\begin{equation*}
\begin{aligned}
|\langle \phi,f \rangle_{H^{s}_{a,\theta}}|=\lim_{t\to 0^{+}}|\langle u(t),f \rangle_{H^{s}_{a,\theta}}|\leq & \liminf_{t\to 0^{+}}\|u(t)\|\\
\leq & \limsup_{t\to 0^{+}}\|u(t)\|\leq \limsup_{t\to 0^{+}} \rho(t)^{\frac{1}{2}}\leq \|\phi\|_{H^{s}_{a,\theta}},  
\end{aligned}    
\end{equation*}
where we have used the fact that $\lim_{t\to 0^{+}}\rho(t)=\|\phi\|_{H^{s}_{a,\theta}}^2$ (see \eqref{parabolicUniformbound}). Consequently, since $f$ is arbitrary, we deduce $\lim_{t\to 0^{+}}u(t)=\phi$ in $H^s_{a,\theta}(\mathbb{R})$, i.e., $u$ is continuous to the right of $t=0$. Given $t'\in [0,T]$, let $v$ be the weak solution of \eqref{gbo} with initial condition $u(t')$ (such  a solution exists by steps (i)-(iv)). It follows from the uniqueness results in (iv) that $v(t)=u(t+t')$, and by the previous argument, we have $v(t)$ is continuous to the right at the origin, it follows that $u$ is continuous to the right at $t=t'$, i.e., since $t'$ is arbitrary, $u$ is continuous to the right in $[0,T]$. Finally, since the equation in \eqref{gbo} is invariant under the transformation $(x,t)\rightarrow (-x,t'-t)$, we deduce continuity to the left at $t'$ from uniqueness, and continuity to the right of weak solutions of \eqref{gbo}. Summarizing $u\in C([0,T];H^s_{a,\theta}(\mathbb{R}))$, which completes the proof.
\end{proof}


\section*{Acknowledgment}

A.P. is partially supported by Conselho Nacional de Desenvolvimento Científico e Tecnológico - CNPq grant 309450/2023-3.
\\ \\
{\bf Data Availability.} 
Data sharing not applicable to this article as no datasets were generated or analyzed during the current study.


\end{document}